\documentclass[12pt,a4paper,reqno]{amsart}
\usepackage[utf8]{inputenc}
\usepackage[T1]{fontenc}
\usepackage{amsmath}
\usepackage{amsthm}
\usepackage{amssymb}
\usepackage[abbrev,lite]{amsrefs}
\usepackage{mathrsfs}
\usepackage[dvipsnames]{xcolor}
\usepackage{bm}
\usepackage{enumitem}
\usepackage{hyperref}
\usepackage{graphicx}
\DefineSimpleKey{bib}{myurl}
\newcommand\myurl[1]{\url{#1}}
\BibSpec{webpage}{%
+{}{\PrintAuthors} {author}
+{,}{ \textit} {title}
+{}{ \parenthesize} {date}
+{,}{ \myurl} {myurl}
+{,}{ } {note}
+{.}{ } {transition}
}
\AtBeginDocument{\def\MR#1{}}
\makeatletter
\@namedef{subjclassname@2020}{%
\textup{2020} Mathematics Subject Classification}
\makeatother
\numberwithin{equation}{section}
\allowdisplaybreaks
\newtheorem{thm}{}[section]
\newtheorem{theorem}[thm]{Theorem}
\newtheorem{corollary}[thm]{Corollary}
\newtheorem{lemma}[thm]{Lemma}
\newtheorem{proposition}[thm]{Proposition}
\theoremstyle{definition}

\newtheorem{definition}[thm]{Definition}
\newtheorem{question}[thm]{Question}
\newcommand{\Id}{\ensuremath{\mathrm{Id}}}
\newcommand{\blambda}{\ensuremath{\bm{\lambda}}}
\newcommand{\ee}{\ensuremath{\bm{e}}}

\newcommand{\xx}{\ensuremath{\bm{x}}}
\newcommand{\yy}{\ensuremath{\bm{y}}}

\newcommand{\UU}{\ensuremath{\mathbb{U}}}
\newcommand{\YY}{\ensuremath{\mathbb{Y}}}
\newcommand{\XX}{\ensuremath{\mathbb{X}}}
\newcommand{\VV}{\ensuremath{\mathbb{V}}}

\newcommand{\RR}{\ensuremath{\mathbb{R}}}
\newcommand{\FF}{\ensuremath{\mathbb{F}}}
\newcommand{\NN}{\ensuremath{\mathbb{N}}}
\newcommand{\ZZ}{\ensuremath{\mathbb{Z}}}
\newcommand{\MM}{\ensuremath{\mathbb{M}}}
\newcommand{\JJ}{\ensuremath{\mathbb{J}}}
\newcommand{\Sym}{\ensuremath{\mathbb{S}}}
\newcommand{\HH}{\ensuremath{\mathbb{H}}}
\newcommand{\GG}{\ensuremath{\mathbb{G}}}

\newcommand{\UB}{\ensuremath{\mathcal{U}}}
\newcommand{\VB}{\ensuremath{\mathcal{V}}}

\newcommand{\Mt}{\ensuremath{\mathcal{M}}}
\newcommand{\Nt}{\ensuremath{\mathcal{N}}}
\newcommand{\Dt}{\ensuremath{\mathcal{D}}}
\newcommand{\Ht}{\ensuremath{\mathcal{H}}}
\newcommand{\Et}{\ensuremath{\mathcal{E}}}
\newcommand{\XB}{\ensuremath{\mathcal{X}}}
\newcommand{\YB}{\ensuremath{\mathcal{Y}}}

\newcommand{\Jt}{\ensuremath{\mathcal{J}}}
\newcommand{\Kt}{\ensuremath{\mathcal{K}}}

\newcommand{\LL}{\ensuremath{\bm{L}}}
\newcommand{\sue}{\ensuremath{\bm{\sigma}}}
\newcommand{\ile}{\ensuremath{\bm{\tau}}}
\newcommand{\Ft}{\ensuremath{\mathcal{F}}}
\newcommand{\At}{\ensuremath{\mathcal{A}}}
\newcommand{\Bt}{\ensuremath{\mathcal{B}}}
\newcommand{\St}{\ensuremath{\mathcal{S}}}
\newcommand{\Ts}{\ensuremath{\mathcal{T}}}
\newcommand{\Rt}{\ensuremath{\mathcal{R}}}
\newcommand{\It}{\ensuremath{\mathcal{I}}}
\newcommand{\KK}{\ensuremath{\bm{K}}}
\newcommand{\TT}{\ensuremath{\bm{T}}}
\DeclareMathOperator{\codim}{codim}
\DeclareMathOperator{\supp}{supp}
\DeclareMathOperator{\diag}{diag}
\DeclareMathOperator{\Ker}{Ker}
\DeclareMathOperator*{\Ave}{Ave}
\DeclareMathOperator{\ind}{ind}
\DeclareMathOperator{\dig}{\bm{\delta}}
\DeclareMathOperator{\mat}{\bm{\mu}}
\newcommand\subsetsim{\mathrel{%
\ooalign{\raise0.2ex\hbox{$\subset$}\cr\hidewidth\raise-0.8ex\hbox{\scalebox{0.9}{$\sim$}}\hidewidth\cr}}}
\newcommand{\abs}[1]{\left\lvert#1\right\rvert}
\newcommand{\norm}[1]{\left\lVert#1\right\rVert}

\newcommand{\enbrace}[1]{\left\lbrace#1\right\rbrace}

\newcommand{\enpar}[1]{\left(#1\right)}
\author[F. Albiac]{F. Albiac}
\address{Department of Mathematics, Statistics and Computer Sciences, and InaMat$^{2}$\\ Universidad P\'ublica de Navarra\\
Pamplona 31006\\ Spain}
\email{fernando.albiac@unavarra.es}
\author[J. L. Ansorena]{J. L. Ansorena}
\address{Department of Mathematics and Computer Sciences\\
Universidad de La Rioja\\
Logro\~no 26004\\ Spain}
\email{joseluis.ansorena@unirioja.es}
\begin{document}
\title[]{Unconditional structure of Banach spaces with few operators}
\begin{abstract}
This article was initially motivated by our goal to show that the Banach space $\GG$ constructed by Gowers in \cite{Gowers1994} to settle Banach's hyperplane problem has a unique unconditional basis. This uniqueness result served as a springboard to ask whether further structural insights could be derived by rigging Gowers original construction. As it turned out, the $p$-convexification of $\GG$ for $1< p<\infty$, $p\not=2$, provides a family of Banach spaces, each of them with a unique unconditional basis containing block bases whose spreading models are not equivalent to the unit vector basis of $\ell_1$, $\ell_2$, or $c_0$. This solves in the negative a forty year-old open problem raised by Bourgain, Casazza, Lindenstrauss, and Tzafriri in their 1985 \emph{Memoir} \cite{BCLT1985}, where they studied the uniqueness of unconditional structure in infinite direct sums of those three spaces with the aim to classify all Banach spaces with a unique unconditional basis. As a by-product of our work, we also disprove the conjecture in structure theory that a space having a unique unconditional basis must be isomorphic to its square, and evince that when a Banach space $\XX$ with an unconditional basis has few operators then the space itself and all its complemented subspaces have a unique unconditional structure.
\end{abstract}
\thanks{The research reported in this paper was partially conducted during the authors' visit at the Erwin Schr\"odinger International Institute for Mathematics and Physics in Vienna, on the occasion of the workshop \emph{Structures in Banach spaces}. The authors acknowledge the support of the Spanish Ministry for Science and Innovation under Grant PID2022-138342NB-I00 for \emph{Functional Analysis Techniques in Approximation Theory and Applications}.}
\maketitle
\section{Introduction}\noindent
Unconditional bases have played a significant role in the study of Banach spaces and their applications in analysis since the dawn of the theory. The existence of a Schauder basis with no special properties does not give much information on the structure of the space. However, assuming the existence of an unconditional basis may lead to significant improvements since it normally facilitates the study of geometric and topological features of the space such as
order, approximation properties, type and cotype, reflexivity, finite-dimensional decompositions, existence of operators of some kind, complemented subspace structure, existence of embeddings, or finite representability of $\ell_p$-spaces, just to name a few. The combination of having an unconditional basis with other structural properties provides a powerful toolkit for analyzing Banach spaces and helps connect the spaces with one another. One can find plenty of instances that show the richness of these interactions, and perhaps James articles from the 1950's are some of the examples that best illustrate the synergies that arise when additional structural features intermingle with the unconditionality of a basis.

One of these properties is precisely the uniqueness (up to equivalence) of unconditional basis. Banach spaces with this property are rare. In fact, there exist three and only three infinite-dimensional Banach spaces $\XX$ in which all normalized unconditional bases are equivalent, namely $c_{0}$, $\ell_{1}$, or $\ell_{2}$, in which case the basis must be symmetric (see \cite{LinZip1969}). The symmetry of the basis implies the self-similarity of the space, in the sense that $\XX$ is isomorphic to its square, and so by induction to the direct sum of finitely-many copies of itself of any size.

For spaces with an unconditional basis which is not symmetric, one can go a little bit further and consider instead normalized bases that become equivalent by a suitable permutation of their elements. The first result along these lines was obtained by Edelstein and Wojtaszczyk \cite{EdWo1976}, who proved that the finite direct sums of the three spaces with a unique unconditional basis have a unique unconditional basis \emph{up to permutation}.

The study of this research topic reached an important landmark in 1985 with the publication of the influential paper \cite{BCLT1985} by Bourgain et al. There the authors investigated the uniqueness of unconditional basis of the infinite direct sums of $c_{0}$, $\ell_{1}$, and $\ell_{2}$ and proved that while $\ell_{1}(c_{0})$, $\ell_{1}(\ell_{2})$, $c_{0}(\ell_{1})$ $c_{0}(\ell_{2})$ do have a unique unconditional basis, the space $\ell_{2}(c_{0})$ and its dual $\ell_{2}(\ell_{1})$ do not. Quite unexpectedly they also found an exotic Banach space, namely the $2$-convexified Tsirelson's space $\Ts^{(2)}$, which also had a unique unconditional basis despite the fact that it was not built out as a direct sum of any of the only three classical sequence spaces with a unique unconditional basis.

The problems left open in the \emph{Memoir} served as stepping stone for pushing forward the research on a subject that had received relatively little attention until then. Casazza and Kalton \cites{CasKal1998,CasKal1999} proved in 1998 that Tsirelson's space $\Ts$ has a unique unconditional basis but that $c_{0}(\Ts)$ does not, thus solving two problems raised in \cite{BCLT1985}. In particular, they provided a negative answer to the general question of whether the direct sum in the sense of $c_{0}$ of a space that has a unique unconditional basis also does. Besides, they found a Banach space with a unique unconditional basis which has a complemented subspace with an unconditional basis that is not unique, solving this way another problem from \cite{BCLT1985}. Apart from these three, all the questions that were raised by Bourgain et al.\@ in the \emph{Memoir} remained open as of today.

After a few more recent additions to the list of spaces with a unique unconditional basis up to permutation \cites{AlbiacAnsorena2022PAMS, AlbiacAnsorena2022b, AAW2022}, it seems clear that it is hopeless to aim at classifying all Banach spaces with this property, and the lack of a unified pattern in the scattered spaces which have a unique unconditional basis up to permutation makes one wonder about the existence, at least, of a class or family of Banach spaces which share the property. On the other hand, in relation to the self-similarity property we referred to before, while all the spaces with unique unconditional basis that we found through the existing literature are also isomorphic to their squares, it is a priory theoretically possible for a space to have a unique unconditional basis and not be isomorphic to its square. As of today, no explicit examples of such spaces were currently known and their existence remained an open problem in functional analysis. In this article we discover an entire class of spaces that fulfil both requirements, which provides new insights into the structure theory of Banach spaces.

To that end, after the preliminary Section~\ref{sect:prelim}, where we gather the terminology and set the mood of the article, in Section~\ref{sect:SSO} we discuss the case of strictly singular operators in spaces with bases. Our structural results rely on the existence of Banach spaces with an unconditional basis, with a small space of operators defined on them. This somehow odd occurrence is examined through the notion of diagonal plus strictly singular operators and bases. In Section~\ref{sect:D+S}, we give general properties of such Banach spaces. Section~\ref{sect:main} contains the main results of the paper. We start by $p$-convexifying Gowers space $\GG$ from \cite{Gowers1994hyp}. We then show that all the new spaces we obtain (including $\GG$) have a unique unconditional basis, contain spreading models other than $c_{0}$, $\ell_{1}$, and $\ell_{2}$, and fail to be isomorphic to their square.
We close the paper in Section~\ref{sect:open} with a few problems that emerge in relation to the research reported here.
\section{Terminology and preparatory results}\label{sect:prelim}\noindent
Let $\XX$ and $\YY$ be Banach spaces over the real or complex field $\FF$. We denote by $\Bt(\XX,\YY)$ the Banach space of all bounded linear operators from $\XX$ into $\YY$, and by $\Bt(\XX)=\Bt(\XX,\XX)$ the Banach algebra of all endomorphisms of $\XX$. We use the symbol $\XX\simeq\YY$ to mean that $\XX$ and $\YY$ are \emph{isomorphic}. We say that $\XX$ \emph{complementably embeds} into $\YY$, and put $\XX\trianglelefteq\YY$, if $\XX$ is isomorphic to a complemented subspace of $\YY$, i.e., there are $J\in\Bt(\XX,\YY)$ and $P\in\Bt(\YY,\XX)$ such that $P\circ J=\Id_\XX$.

We say that two subspaces $\YY$ and $\UU$ of a Banach space $\XX$ are \emph{congruent} (within in $\XX$) if there is an automorphism $T$ of $\XX$ such that $T(\YY)=\UU$. If $\YY$ is complemented, then it is congruent to $\UU$ if and only if $\UU$ is complemented and both the spaces and their complements are isomorphic. Suppose that $c_\YY:=\codim(\YY;\XX)<\infty$ and $c_\UU:=\codim(\UU;\XX)<\infty$. Since any closed subspace of finite codimension is complemented, $\YY$ and $\UU$ are congruent if and only if $c_\YY=c_\UU$.

Given a family $\XB= (\xx_n)_{n\in\Nt}$ in a Banach space $\XX$, we denote by $[\XB]$ or $[\xx_n \colon n\in\Nt]$ its closed linear span. More generally, given $\At\subseteq\Nt$ we set $[\XB|\At]=[\xx_n \colon n\in \At]$. If $[\XB]=\XX$, we say that $\XB$ is \emph{complete} in $\XX$. If
\[
\inf_{n\in\Nt} \norm{\xx_n}>0,\quad \sup_{n\in\Nt} \norm{\xx_n}<\infty,
\]
we say that $\XB$ is \emph{semi-normalized}.

A \emph{K\"othe space} $\KK$ over a measure space will be a Banach lattice relative to the usual order on the set of measurable functions. If the simple integrable functions are dense in $\KK$, we say that $\KK$ is \emph{minimal}. An \emph{atomic lattice} will be a K\"othe space over the counting measure on a countable set. If $\KK$ is an atomic lattice over $\Nt$ and $\Mt\subseteq\Nt$, we denote by $\KK[\Mt]$ the restriction of $\KK$ to $\Mt$. We say that the atomic lattice $\KK$ is \emph{normalized} if the unit vector system $(\ee_n)_{n\in\Nt}$ is as semi-normalized family of $\KK$. Some examples of normalized atomic lattices are $\ell_\infty(\Nt)$, $c_0(\Nt)$ and the Lebesgue spaces $\ell_p(\Nt)$, $1\le p<\infty$, modelled after a countable set $\Nt$. Apart from the first one, the rest are minimal.

A \emph{complete minimal system} in a Banach space $\XX$ is a complete sequence $(\xx_n)_{n\in\Nt}$ in $\XX$ for which there exists a sequence $\XB^*=(\xx_n^*)_{n\in\Nt}$ in $\XX^{\ast}$ such that
\[
\xx_k^*(\xx_n)=\delta_{k,n},
\]
where $\delta_{k,n}=1$ if $k=n$ and $\delta_{k,n}=0$ otherwise.
We say that $\XB^*$ is the \emph{dual family} or dual basis of $\XB$. If the \emph{coefficient transform} relative to $\XB$
\[
x \mapsto (\xx_n^*(x))_{n\in\Nt}, \quad x\in\XX,
\]
is one-to-one, we say that $\XB$ is a \emph{Markushevich basis}. The support of $x\in\XX$ relative to $\XB$ is the set
\[
\supp(x)=\enbrace{x\in\XX \colon \xx_n^*(x)\not=0}.
\]
Let $\At\subseteq\Nt$ and $\beta=(b_n)_{n\in \At}\in\FF^{\At}$. If there is a bounded linear operator
\[
M:=M_\beta[\XB]\in \Bt([\XB|\At],\XX)
\]
such that $M(\xx_n)=b_n\, \xx_n$ for all $n\in\At$, we say that $\beta$ is a multiplier relative to $\XB$. If $\chi_A$ is a multiplier we denote by
\[
S_A[\XB]=M_{\chi_A}[\XB]
\]
the canonical projection onto $[\XB|A]$ relative to $\XB$. If $\Nt=\NN$, for each $n\in\NN$ we put
\[
S_n[\XB] =S_A[\XB], \quad A=[1,n]\cap\ZZ.
\]

A sequence $\XB=(\xx_n)_{n=1}^\infty$ in $\XX$ is said to be \emph{a Schauder basis} of $\XX$ if for every $x\in\XX$ there is a unique sequence $(a_n)_{n=1}^\infty$ in $\FF$ such that the series $\sum_{n=1}^\infty a_n\, \xx_n$ converges to $x$ (in norm). A family $\XB= (\xx_n)_{n\in\Nt}$ in $\XX$ ($\Nt$ countable) is an \emph{unconditional basis} of $\XX$ if for every $x\in\XX$ there is a unique family $(a_n)_{n\in\Nt}$ in $\FF$ such that the series $\sum_{n\in\Nt} a_n\, \xx_n$ converges to $x$ unconditionally. Unconditional basis and Schauder basis are particular cases of Markushevich bases, so on occasion we will use this more general type of bases for making statements that are valid for both types of bases. In fact, Schauder bases are those complete minimal systems $\XB$ modelled on $\NN$ such that there is a constant $K$ such that $\norm{S_J[\XB]}\le K$ for all finite integer intervals $J$. In turn, unconditional bases are those complete minimal systems $\XB$ such that every bounded sequence is a multiplier. Besides, if $\XB$ is an unconditional basis modelled on a set $\Nt$, then there is a constant $C\ge 1$ such that $\norm{M_\beta[\XB]}\le C \norm{\beta}_\infty$ for all $\At\subseteq\Nt$ and all $\beta=(b_n)_{n\in \At}\in\ell_\infty(\At)$. Conversely, if $\XB$ is a complete minimal system modelled after $\Nt$ for which there is a constant $C$ such that
\begin{equation}\label{eq:SupUnc}
\norm{S_{A}[\XB]}\le C, \quad A\subseteq\Nt,
\end{equation}
then $\XB$ is an unconditional basis. The optimal constant $C$ in \eqref{eq:SupUnc} is called the \emph{suppression unconditional constant} of $\XB$.

If $\XB$ is a complete minimal system whose dual basis $\XB^*=(\xx_n^*)_{n\in\Nt}$ is bounded, then the coefficient transform defines a bounded linear operator from $\XX$ into $ c_0(\Nt)$. Semi-normalized Schauder bases and semi-normalized unconditional bases have this property.

A Schauder basis $\XB=(\xx_n)_{n=1}^\infty$ of a Banach space $\XX$ is said to be \emph{shrinking} if its dual basis complete sequence of the dual space. If there is $x\in\XX$ whose coefficients relative to $\XB$ are $(a_n)_{n=1}^\infty$ as long as
\[
\sup_{m\in\NN} \norm{\sum_{n=1}^m a_n \, \xx_n} <\infty,
\]
we say that $\XB$ is \emph{boundedly complete}.

Given families $\XB=(\xx_n)_{n\in\Nt}$ and $\YB=(\yy_n)_{n\in\Mt}$ in Banach spaces $\XX$ and $\YY$, respectively, we say that $\XB$ and $\YB$ are \emph{permutatively equivalent}, and put $\XB\sim \YB$, if there is a bijection $\pi\colon \Nt \to \Mt$ and an isomorphic embedding $T\colon[\XB]\to \YY$ such that $T(\xx_n)=\yy_{\pi(n)}$ for all $n\in\Nt$. If $\pi$ is the identity map, $\XB$ and $\YB$ are said to be \emph{equivalent}. If keeping track of the constants is relevant, we say that $\XB$ and $\YB$ are $C$-equivalent provided that
\[
\max\enbrace{\norm{T},\norm{T^{-1}}}\le C.
\]
If $\XB$ is equivalent to a sub-family of $\YB$, that is, the map $\pi$ is just a one-to-one map instead of a bijection, we say that $\XB$ is \emph{sub-equivalent} to $\YB$, and we put $\XB\subsetsim \YB$. If $\XB\subsetsim \YB$ and $\YB$ is unconditional and/or semi-normalized, so is $\XB$. A sub-family of an unconditional basis $\YB$ is called a \emph{subbasis} of $\YB$.

We denote by $\sqcup_{n\in \Nt} \At_n$ the \emph{disjoint union} of a family $(\At_n)_{n\in\Nt}$, i.e,
\[
\Rt:=\sqcup_{n\in \Nt} \At_n=\cup_{n\in \Nt} \enpar{\At_n\times\{n\}}.
\]
Let $\KK$ be an atomic lattice over $\Nt$, and for each $n\in\Nt$ let $\KK_n$ be an atomic lattice over $\At_n$. Then
\[
\enpar{\bigoplus_{n\in \Nt} \KK_n}_{\KK}
\]
is an atomic lattice over $\Rt$. Suppose that $\Nt$ is finite and that $\XB_n=(\xx_{n,a})_{a\in\At_n}$ is a family in a Banach space $\XX_n$ for each $n\in\Nt$. Let $L_n$ be the canonical embedding of $\XX_n$ into the Banach space $\XX=\oplus_{n\in\Nt} \XX_n$. Set
\[
\XB=\bigoplus_{n\in \Nt} \XB_n=\enpar{L_n(\xx_{a,n})}_{(a,n)\in\Rt}.
\]
If $\XB_n$ is unconditional, complete, or semi-normalized for all $n\in\Nt$, so is $\XB$. If $\Nt=\NN\cap[1,m]$ for some $m\in\NN$, and there is a family $\XB$ such that $\XB_n=\XB$ for all $n\in\Nt$, we put $\UB=\XB^m$.

By a \emph{block sequence} relative to a sequence $\XB:=(x_n)_{n=1}^\infty$ in a Banach space we mean a sequence $\YB=(y_k)_{k=1}^\eta$, $\eta\in\NN\cup\{\infty\}$, for which there is an increasing sequence of integers $(n_k)_{k=0}^\eta$ with $n_0=0$ such that each vector $y_k$ is a linear combination of $(x_n)_{n=1+n_{k-1}}^{n_k}$. Tipically, $\XB$ is a Schauder basis and $\YB$ consists of non-null vectors, in which case we say that $\YB$ is a \emph{block basic sequence}.

We will use several types of finite representability. We say that $\ell_p$, $1\le p\le \infty$, is \emph{crudely finitely representable} in a Banach space $\XX$ if there is $C\ge 1$ such that for every $n\in\NN$ and every $\varepsilon>0$ there is an $n$-tuple $\XB_n=(x_k)_{k=1}^n$ which is $(C+\varepsilon)$-equivalent to the unit vector basis of $\ell_p^n$. If each $\XB_n$ is a block sequence of a given sequence $\YB:=(y_k)_{k=1}^\infty$, we say that $\ell_p$ is \emph{crudely finitely block representable} relative to $\YB$. If $\XX$ is a lattice and $\XB_n$ is pairwise disjointly supported for each $n\in\NN$, we say that $\ell_p$ is \emph{crudely disjointly finitely representable}. If for each $n\in\NN$ the operator $T_n\colon [\XB_n] \to \ell_p^n$ that witnesses the equivalence extends to an operator $P_n\colon \XX\to \ell_p^n$ with
\[
\norm{P_n} \norm{T_n^{-1}} \le C+\varepsilon,
\]
we say that $\ell_p$ is \emph{crudely finitely disjointly complementably representable} in $\XX$. In all cases, if we can choose $C=1$ we drop the word `crudely'.

Given a countable infinite set $\Nt$, the convergence of a family $\XB=(x_n)_{n\in\Nt}$ in a topological space does not depend on the existence on an arrangement of $\XB$. In fact, $\lim_{n\in\Nt} x_n=x$ if and only if for every neighbourhood $U$ of $x$ there is $N\subseteq\Nt$ finite such that $x_n\in U$ for all $n\in\Nt\setminus F$. Note that by adopting this definition, every finite family converges. Similarly, if $\XB$ is a family in $[-\infty,\infty]$ and we put
\[
\limsup_{n\in\Nt} x_n=\inf_{\abs{F}<\infty} \sup_{n\in \Nt\setminus F} x_n, \quad \liminf_{n\in\Nt} x_n=\sup_{\abs{F}<\infty} \inf_{n\in \Nt\setminus F} x_n,
\]
with the convention that $\sup\emptyset=-\infty$ and $\inf\emptyset=\infty$, if $\Nt$ is finite then
\[
\limsup_{n\in\Nt} a_n=-\infty, \quad \liminf_{n\in\Nt} a_n=\infty.
\]

We denote by $c_{00}(\Nt)$ the linear space of all eventually null scalar-valued families on $\Nt$, and we put $c_{00}=c_{00}(\NN)$.
\subsection{Uniqueness of unconditional basis and powers of bases}
A Banach space $\XX$ has a \emph{unique unconditional basis} if it has a (semi-normalized) unconditional basis $\XB$, and $\XB$ is the unique unconditional basis of $\XX$ (up to normalization, permutation and equivalence), i.e., $\YB\sim\XB$ whenever $\YB$ is another semi-normalized unconditional basis of $\XX$.

All known spaces $\XX$ with a unique unconditional basis have a (semi-normalized) unconditional basis $\XB$ which is permutatively equivalent to its square. Besides, this feature plays a key role in the proof that $\XX$ has a unique unconditional basis. In hindsight, all the proofs we find in the literature that a given Banach space has a unique unconditional basis share the pattern of the following lemma.

\begin{lemma}[see \cite{AlbiacAnsorena2025}]\label{lem:AAA}
Let $\XX$ be a Banach space with a semi-normalized unconditional basis $\XB$. Suppose that $\XB^2 \sim \XB$ and that for all semi-normalized unconditional bases $\YB$ and $\UB$ of Banach spaces $\YY$ and $\UU$, respectively, such that $\YY\trianglelefteq \UU$ and $\UB\subsetsim \XB$ there is $m\in\NN$ such that $\YB \subsetsim \XB^m$. Then $\XX$ has a unique unconditional basis.
\end{lemma}

However, there is a priori no relation between an unconditional basis being unique and the basis being permutatively equivalent to its square. So the following question has been implicit for a while.

\begin{question}\label{qt:AA}
Is there a Banach space $\XX$ with a unique unconditional basis $\XB$ that fails to be permutatively equivalent to its square?
\end{question}
\subsection{Uniqueness of unconditional basis and optimal lattice estimates}
Recall that a K\"othe space $\KK$ over a measure space $(\Omega,\Sigma,\mu)$ is said to be lattice $p$-convex (resp., $p$-concave), $1\le p\le \infty$, if there is a constant $C$ such that for every finite family $(x_j)_{j\in F}$ in $\KK$ we have
\[
\norm{ \enpar{\sum_{j\in F} \abs{x_j}^p}^{1/p}}\le C\enpar{\sum_{j\in F} \norm{x_j}^p}^{1/p},
\]
(resp.,
\[
\enpar{\sum_{j\in F} \norm{x_j}^p}^{1/p}\le C \norm{ \enpar{\sum_{j\in F} \abs{x_j}^p}^{1/p}}).
\]
If these estimates hold for disjointly supported families, we say that $\KK$ satisfies an upper (resp., lower) $p$-estimate. It is known \cite{LinTza1979} that

\begin{align*}
\sue(\KK) :=&\sup\enbrace{ p\in[1,\infty] \colon \mbox{$\KK$ satisfies an upper $p$-estimate}} \\
=&\sup\enbrace{ p\in[1,\infty] \colon \mbox{$\KK$ is lattice $p$-convex}}, \\
\ile(\KK) :=&\inf\enbrace{ p\in[1,\infty] \colon \mbox{$\KK$ satisfies a lower $p$-estimate}},\\
=&\sup\enbrace{ p\in[1,\infty] \colon \mbox{$\KK$ is lattice $p$-concave}}.
\end{align*}

Given $1\le p < \infty$, the \emph{$p$-convexification}
\[
\KK^{(p)}=\enbrace{f\in L_0(\mu) \colon \abs{f}^p\in \KK}
\]
of $\KK$ satisfies $\sue(\KK^{(p)})=p \, \sue(\KK)$, and $\ile(\KK^{(p)})=p\, \ile(\KK)$.

Despite the fact that convexifying is a somewhat na\"{\i}ve procedure for constructing new Banach spaces from old ones, it should be taken into account that applying this process dramatically alters the linear structure of the spaces. For example, the spaces $L_p(\mu)$, $1<p<\infty$, can be regarded as the $p$-convexifications on $L_1(\mu)$, and indeed they behave quite differently from $L_1(\mu)$. In contrast, there are Banach lattices $\KK$, such as $L_\infty(\mu)$, $c_0$, or the original Tsirelson space $\Ts^*$, such that $\KK^{(p)}=\KK$ for all $p\in[1,\infty)$. Note that such lattices $\KK$ are necessarily lattice $p$-convex for all $p\in[1,\infty)$ so that $\sue(\KK)= \ile(\KK) = \infty$. Based on the observed patterns we are led to expect $\KK$ to have a unique unconditional basis only when
\[
\{ \sue(\KK), \ile(\KK)\} \subseteq \It := \{1,2,\infty\}.
\]
For instance, with the convention that by $\ell_\infty$ we mean $c_0$, we can observe that the spaces $\ell_p$ for $1\le p \le \infty$ have a unique unconditional basis if and only if $p=\sue(\ell_p)=\ile(\ell_p)\in \It$ (see \cites{KotheToeplitz1934,LinPel1968,Pel1960}). Similarly, the $p$-convexified Tsirelson space $\Ts^{(p)}$, $1\le p<\infty$ has a unique unconditional basis if and only if $p=\sue(\Ts^{(p)})=\ile(\Ts^{(p)})\in \It$ (see \cites{BCLT1985,CasKal1998,AlbiacAnsorena2024c}). The original Tsirelson space $\Ts^*$ has a unique unconditional basis as well, and as we just pointed out, $\sue(\Ts^*)= \ile(\Ts^*) = \infty$.

Let us have a look now at what happens with the aforementioned mixed-norm sequence spaces. Following \cite{BCLT1985}, we put $Z_{p,q}=\ell_q(\ell_p)$ for $p$, $q\in[1,\infty]$, $p\not=q$. These spaces are atomic lattices, and $\sue(Z_{p,q})=\min\{p,q\}$. These spaces have a unique unconditional basis if and only if $p=1$ and $q=\infty$, $p=\infty$ and $q=1$, or $p=2$ and $q\in\{1,\infty\}$ (see \cites{BCLT1985}). Note that in all those cases we have $\sue(Z_{p,q})=\min\{p,q\}\in \It$ and $\ile(Z_{p,q})=\max\{p,q\}\in \It$. Based on these examples, the following question naturally arises.

\begin{question}\label{qt:UTAPCloselp}
Is there an atomic lattice $\KK$ with a unique unconditional basis such that the intersection set
\[
\enbrace{ \sue(\KK) , \ile(\KK) } \cap \enpar{(1,2)\cup(2,\infty)}
\]
is nonempty?
\end{question}

Set $p=\sue(\KK)$ (resp., $p=\ile(\KK)$). Schep \cite{Schep1992} proved that $\ell_{p}$ is finitely disjointly representable in $\KK$. If the supremum (resp., infimum) defining $p$ is attained, the following (novel) version of Krivine's theorem yields complemented copies of $\ell_p^n$'s in $\KK$.

\begin{lemma}\label{lem:KrivineA}
Suppose a Banach lattice $\KK$ satisfies an upper $p$-estimate ($1\le p\le \infty$) with constant $C$. Let $(x_k^*)_{k=1}^n$ be a pairwise disjointly supported sequence in $\KK^*$ which is $D$-equivalent to the canonical basis of $\ell_q^n$, where $q$ is the conjugate exponent of $p$. Then for any $E>D$ there are linear operators $J\colon \ell_p^n \to\KK$ and $Q\colon \KK \to \ell_p^n$ such that $P\circ S=\Id_{\ell_p^n}$, $\norm{J} \norm{P}\le E$, and $\supp(J(\ee_k))=\supp(\xx_k^*)$ for all $k=1$, \dots, $n$.
\end{lemma}

\begin{lemma}\label{lem:KrivineB}
Suppose a Banach lattice $\KK$ satisfies a lower $p$-estimate ($1\le p\le \infty$) with constant $C$. Let $\XB=(x_k)_{k=1}^n$ be a pairwise disjointly supported sequence in $\KK$ which is $D$-equivalent to the canonical basis of $\ell_p^n$. Then there are linear operators $J\colon \ell_p^n \to\KK$ and $Q\colon \KK \to \ell_p^n$ such that $P\circ S=\Id_{\ell_p^n}$, $\norm{J} \norm{P}\le CD$, and $J(\ee_k)=\xx_k$ for all $k=1$, \dots, $n$.
\end{lemma}

\begin{proof}[Proof of Lemmas \ref{lem:KrivineA} and \ref{lem:KrivineB}]
Fix $k=1$, \dots, $n$. To prove Lemma \ref{lem:KrivineA}, we pick $x_k\in\KK$ for $x_k^*$ with $\xx_k^*(\xx_k)=1$, $\supp(k_k)=\supp(k_k^*)$, and $\norm{x_k}$ close enough to one. To prove Lemma \ref{lem:KrivineB}, we pick a norming functional $x_k^*$ for $x_k$. Let $J\colon \ell_p^n \to \KK$ and $Q\colon\ell_q^n\to \KK^*$ be defined by $J(\ee_k)=\xx_k$ and $Q(\ee_k)=\xx_k^*$ for all $k=1$, \dots, $n$. Let $P\colon \KK \to \ell_p^n$ be the restriction to $\KK$ of the dual operator of $Q$. Since $P(f)=(\xx_k^*(f))_{k=1}^n$ for all $f\in \KK$, $P\circ S=\Id_{\ell_p^n}$. When proving Lemma \ref{lem:KrivineA}, $\norm{P}\le D$, and we can choose $(x_k)_{k=1}^n$ such that $\norm{J}$ arbitrarily close to $C$. When proving Lemma \ref{lem:KrivineB}, $\norm{J}\le D$, and $\norm{P}\le C$ by \cite{LinTza1979}*{Proposition 1.f.5}.
\end{proof}

\begin{theorem}\label{thm:Complp}
Let $\KK$ be a Banach lattice and $p\in[1,\infty]$. Suppose that either
\begin{itemize}[leftmargin=*]
\item $\KK$ satisfies an upper $p$-estimate and fails to satisfy an upper $r$-estimate for any $r>p$, or
\item $\KK$ satisfies a lower $p$-estimate and fails to satisfy a lower $r$-estimate for any $r<p$.
\end{itemize}
Then $\ell_p$ is crudely finitely disjointly complementably representable in $\KK$.
\end{theorem}

\begin{proof}
Let $q$ be the conjugate exponent of $p$. In the first case, by \cite{LinTza1979}*{Proposition 1.f.5}, $\KK^*$ satisfies a lower $q$ estimate, and $\ile(\KK^*)=q$. Hence, $\ell_q$ is disjointly finitely representable in $\KK^*$, and the result follows from Lemma~\ref{lem:KrivineA}. In case the second assumption holds, $\ell_p$ is disjointly finitely representable in $\KK$, so the result follows from Lemma~\ref{lem:KrivineB}.
\end{proof}

Let $1<p<\infty$. We note that Pe{\l}czy\'{n}ski's uniform isomorphisms from \cite{Pel1960},
\[
\ell_p^{2^n-1} \simeq \enpar{ \bigoplus_{j=0}^{n-1} \ell_2^{2^j}}_{\ell_p}, \quad n\in\NN,
\]
prevent $\ell_p$ from having a unique unconditional basis unless $p=2$. So, Theorem~\ref{thm:Complp} relates
Question~\ref{qt:UTAPCloselp} to the following one.

\begin{question}\label{qt:lpFC}
Let $p\in (1,2)\cup(2,\infty)$. Is $\ell_p$ crudely finitely complementably representable in some Banach space $\XX$ with a unique unconditional basis?
\end{question}
\subsection{Spreading models of Banach spaces with a unique unconditional basis}
A \emph{subsymmetric sequence space} will be a Banach space $\Sym\subseteq\FF^\NN$ for which the unit vector system $(\ee_n)_{n=1}^\infty$ is a normalized of $1$-suppression unconditional basis isometrically equivalent to $(\ee_{\varphi(n)})_{n=1}^\infty$ for every increasing map $\varphi\colon\NN\to\NN$. It is known \cite{BL1983} that for any weakly null normalized sequence $\XB=(x_n)_{n=1}^\infty$ in any Banach space there is subsequence $\XB=(x_n)_{n=1}^\infty$ and a subsymmetric sequence space $(\Sym, \norm{\cdot}_\Sym)$ such that
\[
\norm{s}_\Sym=\lim_{(n_j)_{j=1}^m \in \NN^{(m)}} \norm{\sum_{j=1}^m a_j \, x_{n_j}}
\]
for every $s=\sum_{j=1}^m a_j\, \ee_j \in c_{00}$. Such a sequence $\XB$ is said to be a \emph{good sequence}, and $\Sym$ is said to be the \emph{unconditional spreading model} of $\XX$ associated with $\XB$.

The study of the (unconditional) spreading model structure of Banach spaces gained relevance in functional analysis after the discovery that the original Tsirelson space $\Ts^*$ contains no Banach space isomorphic to a subsymmetric sequence space
(see \cites{Tsirelson1974, CasShu1989}). Thus, when Bourgain et al.\@ found a space with a unique unconditional basis which is not built out as a direct sum of $c_0$, $\ell_1$, and $\ell_2$ (namely the $2$-convexified Tsirelson space $\Ts^{(2)}$), based on their results in the \emph{Memoir} they conjectured that the unique unconditional spreading models of Banach spaces with a unique unconditional basis would be $c_0$, $\ell_1$ or $\ell_2$. Other Banach spaces whose uniqueness of unconditional basis was established later on, backed up Bourgain et al.'s prediction. However, the following general problem remained unanswered.

\begin{question}[\cite{BCLT1985}*{Problem 11.6}]\label{qt:BCLT}
Assume $\XB$ is an unconditional basis of a Banach space $\XX$ having a unique unconditional basis. Is every unconditional spreading model of $\XX$ equivalent to the unit vector basis of $\ell_1$, $\ell_2$, or $c_0$?
\end{question}

To make it clear that none of the known spaces so far with a unique unconditional basis provides a negative answer to Question~\ref{qt:BCLT}, we make a detour en route to briefly discuss the spreading model structure of those spaces. Let us first state a well-known result that follows by combining the principle of small perturbations with the gliding hump technique.

\begin{lemma}\label{lem:SMBBS}
Let $\Sym$ be an unconditional spreading model of a Banach space $\XX$ with a Schauder basis $\XB$. Then $\Sym$ is associated with a block basic sequence relative to $\XB$.
\end{lemma}

Given $1\le p\le\infty$, we say that a Schauder basis $(\xx_n)_{n=1}^\infty$ is an \emph{asymptotic $\ell_p$-basis} if there is a constant $C$ such that for all $m\in\NN$ there is $k=k(m)\in\NN$ such that every $m$-tuple block relative to $(\xx_n)_{n=k}^\infty$ is $C$-equivalent to the unit vector basis of $\ell_p^m$.
By Lemma~\ref{lem:SMBBS}, the space $\ell_p$ ($c_0$ if $p=\infty$) is the unique spreading model of a Banach space with an asymptotic $\ell_p$-basis.

The canonical basis of Tsirelson's space $\Ts$ is an asymptotic $\ell_1$-basis \cite{CasShu1989}, whence the canonical basis of Tsirelson's space $\Ts^{(2)}$ is an asymptotic $\ell_2$-basis. By duality (see Lemma~\ref{lem:KrivineA}), the canonical basis of the original Tsirelson's space $\Ts^*$ is an asymptotic $\ell_\infty$-basis, and the canonical basis of $(\Ts^{(2)})^*$ is an asymptotic $\ell_2$-basis.

Given a nonincreasing function
\[
\varphi\colon[0,\infty)\to[1,\infty)
\]
with $\lim_{s\to\infty} \varphi(s)=1$ we set
\[
F(t)=\exp\enpar{-\int_0^{-\log(t)} \varphi(s)\, ds}, \quad 0\le t \le 1.
\]
The spaces from \cite{CasKal1998} with a unique unconditional basis which have a complemented subspace with an unconditional basis that is not unique are subspaces of Orlicz sequence spaces $\ell_{F}$, with $F$ as above, spanned by normalized block basic sequences $(x_n)_{n=1}^\infty$ with
\[
\lim_n \norm{x_n}_\infty=0.
\]
These spaces are isometrically isomorphic to Musielak-Orlicz sequence spaces associated with sequences $(F_n)_{n=1}^\infty$ such that
\begin{equation}\label{eq:OrliczNakanoUTAP}
\lim_n F_n(t)=t \mbox{ uniformly on $[0,1/2]$.}
\end{equation}
It can be proved that the canonical bases of Musielak-Orlicz spaces $\ell(F_n)$ with $(F_n)_{n=1}^\infty$ satisfying \eqref{eq:OrliczNakanoUTAP} are asymptotic $\ell_1$-bases.

Cassaza and Kalton \cite{CasKal1998} proved that the Nakano spaces $\ell(p_{n})$ associated with a sequence $(p_n)_{n=1}^\infty$ converging to one or to infinity have a unique unconditional basis in the case when $\ell(p_{n})$ is lattice isomorphic to its square. If $(p_n)_{n=1}^\infty$ converges to one, the spaces $\ell(p_{n})$ are Musielak-Orlicz sequence spaces associated with sequences $(F_n)_{n=1}^\infty$ satisfying \eqref{eq:OrliczNakanoUTAP}, whence their canonical bases are asymptotic $\ell_1$-bases. By duality, the canonical basis of Nakano spaces $\ell(p_{n})$ associated with a sequence $(p_n)_{n=1}^\infty$ converging to infinity are asymptotic $\ell_\infty$-bases.

To describe the spreading model structure of direct sums of spaces with a unique unconditional basis we need some extra machinery. Let $(\KK, \norm{\cdot}_{\KK})$ be an atomic lattice over a coutable set $\Nt$, $\Mt\subseteq\Nt$, $(\Sym_n,\norm{\cdot}_{\Sym_n})_{n\in\Mt}$ be a countable family of subsymmetric sequence spaces, and $(\lambda_n)_{n\in\Mt}\in (0,\infty)^{\Mt}$ belong unit sphere of $\KK[\Mt]$. Then
\[
\Sym:=\enpar{\bigcap_{n\in\Mt} \lambda_n\, \Sym_n}_{\KK}=\enbrace{ \alpha \in\FF^{\NN} \colon \norm{\enpar{\lambda_n \norm{\alpha}_{\Sym_n}}_{n\in\Mt}}_{\KK[\Mt]}<\infty}
\]
is a subsymmetric sequence space. If $\Mt$ is finite, then different lattices $\KK$ and different families $\blambda$ just yield different equivalent norms on the same vector space. So, in this case we simply put $\Sym=\cap_{n\in\Mt} \Sym_n$.

\begin{lemma}\label{lem:SPDirectSum:Z}
Let $(\XX_n)_{n\in\Nt}$ be a family of Banach spaces and $\KK$ be a minimal atomic lattice over $\Nt$. Let $\Sym$ be an unconditional spreading model of associated with weakly null good sequence $(x_j)_{j=1}^\infty$ in
\[
\XX:=\enpar{\bigoplus_{n\in\Nt} \XX_n}_{\KK}.
\]
Let $S_{\Jt}\colon\XX\to\XX$ denote the coordinate projection associated with $\Jt\subseteq\Nt$, and for each $n\in\Nt$, let $P_n\colon\XX\to\XX_n$ be the canonical coordinate map.
Assume that for every $\varepsilon>0$ threre is $\Jt\subseteq\Nt$ finite such that
\[
\enbrace{j\in\NN \colon \norm{x_j-S_{\Jt}(x_j)}>\varepsilon}
\]
is finite. Then there are $\Mt\subseteq\Nt$, $(\lambda_n)_{n\in\Mt}\in (0,\infty)^{\Mt}$ in the unit sphere of $\KK[\Mt]$, and for each $n\in\Mt$ an unconditional spreading model of $\XX_n$ such that $\Sym=\enpar{\cap_{n\in\Mt} \lambda_n\, \Sym_n}_{\KK}$. Besides, if we put $\lambda_n=0$ for all $n\in\Nt\setminus\Mt$, there is a subsequence $(y_j)_{j=1}^\infty$ of $(x_j)_{j=1}^\infty$ such that
\[
\lim_j \norm{P_n(y_j)}=\lambda_n, \quad n\in\Nt.
\]
\end{lemma}

\begin{proof}
Using Cantor's diagonal technique, by passing to a subsequence we can assume that
$\enpar{\norm{P_n(x_j)}}_{j=1}^\infty$ converges to some $\lambda_n\in[0,\infty)$ for each $n\in\Nt$.
Set
\[
\Mt=\enbrace{n\in\Nt \colon \lambda_n\not=0}.
\]
Passing to a further subsequence, again by Cantor's diagonal technique, we can assume that for each $n\in\Mt$
\[
\enpar{\frac{P_n(x_j))}{ \norm{P_n(x_j)}}}_{j=1}^\infty
\]
is a good sequence associated with an unconditional spreading model $\Sym_n$ over $\XX_n$. Pick a null sequence $(\varepsilon_k)_{k=1}^\infty$ in $(0,\infty)$. Passing to a further subsequence, we recursively construct $(\Jt_k)_{n=1}^\infty$ consisting of finite sets increasing to $\Nt$ such that
\[
\norm{(x_j -S_{\Jt_k}(x_j)}\le \varepsilon_k, \quad j\ge k.
\]
Fix $k\in\NN$ and $\alpha=(a_i)_{i=1}^\infty\in c_{00}$. Let $m\in\NN$ be such that $a_i=0$ for $i>m$. There is an increasing $m$-tuple $(j_i)_{i=1}^m$ in $\NN$ such that
\begin{align*}
\abs{ \norm{\alpha}_\Sym - \norm{\sum_{i=1}^m a_i \ x_{j_i}}} &\le \varepsilon_k,\\
\abs{ \lambda_n \norm{\alpha}_{\Sym_n} - \norm{\sum_{i=1}^m a_i\, P_n\enpar{x_{j_i}}}}&\le\frac{\varepsilon_k}{\norm{\sum_{n\in\Jt_k} \ee_n}_{\KK}}, \quad n\in\Jt_k,\\
\norm{x_{j_i} - S_{\Jt_k}(x_{j_i})}&\le \varepsilon_k.
\end{align*}
Summing up,
\[
\abs{ \norm{\alpha}_\Sym - \norm{\sum_{n\in\Jt_k} \lambda_n \norm{\alpha}_{\Sym_n} \ee_n }_{\KK}}\le \enpar{2+\norm{\alpha}_1} \varepsilon_k.
\]
Letting $k$ tend to infinity and $\alpha$ run over $c_{00}$ we are done.
\end{proof}

Lemma~\ref{lem:SPDirectSum:Z} applies in particular to finite families, in which case it gives the following.

\begin{proposition}\label{prop:SPDirectSum:A}
Let $(\XX_n)_{n\in\Nt}$ be a finite family of Banach spaces. If $\Sym$ is an unconditional spreading model of $\oplus_{n\in\Nt} \XX_n$, then there is $\Jt\subseteq\Nt$ and for each $n\in\Jt$ an unconditional spreading model $\Sym_n$ of $\XX_n$ such that $\Sym=\cap_{n\in\Jt} \Sym_n$.
\end{proposition}

We also give a result concerning infinite direct sums.
\begin{proposition}\label{prop:DirectSumB}
Let $p\in[1,\infty]$ and $(\KK_n)_{n\in\Nt}$ be a countable family of minimal atomic lattices satisfying an uniform upper $p$-estimate. Let $\KK$ be a minimal atomic lattice over $\Nt$.
Let $\Sym$ be an unconditional spreading model of
\[
\XX:=\enpar{\bigoplus_{n\in\Nt} \KK_n}_{\KK}.
\]
\begin{enumerate}[label=(\roman*), leftmargin=*, widest=ii]
\item\label{it:SMA} Assume that $\ell_p$ (take $c_0$ if $p=\infty$) is the unique unconditional spreading model of $\KK$. Then either $\Sym=\ell_p$, or there are $\Mt\subseteq\Nt$, $(\lambda_n)_{n\in\Nt}\in(0,\infty)^\Nt$ in the unit sphere of $\KK[\Mt]$, and for each $n\in\Mt$ an unconditional spreading model $\Sym_n$ of $\KK_n$ such that $\Sym=\enpar{\cap_{n\in\Nt} \lambda_n\, \Sym_n}_{\KK}$.
\item\label{it:SMB} Assume that for each $n\in\Nt$ the space $\ell_p$ is the unique unconditional spreading model of $\KK_n$. Then either $\Sym=\ell_p$ or $\Sym$ is a spreading model of $\KK$.
\end{enumerate}
\end{proposition}

\begin{proof}
By Lemma~\ref{lem:SMBBS}, $\Sym$ is associated with a pairwise disjoint sequence $(x_j)_{j=1}^\infty$. Since $\XX$ is an atomic lattice satisfying an upper $p$-estimate, $\ell_p\subseteq\Sym$. By Cantor's diagonal argument we can assume that there is $(\lambda_n)_{n\in\Nt}$ in $[0,\infty)$ such that
\[
\lim_j \norm{P_n(x_j)}=\lambda_n, \quad n\in\Nt,
\]
where $P_n\colon\XX\to\XX_n$ be the canonical coordinate map.
By Lemma~\ref{lem:SPDirectSum:Z}, we can assume that there is $\varepsilon>0$ such that for all $\Jt\subseteq\Nt$ finite and all $j_0\in\NN$ there is $j>j_0$ such that $\norm{S_{\Jt}(x_j)}>c$. Passing to a subsequence, we find a pairwise disjoint sequence $(\Nt_k)_{k=1}^\infty$ in $\Nt$ such that $\norm{y_j}\ge c$ for all $j\in\NN$, where
\[
y_j = S_{\Nt_j}(x_j) .
\]
Besides, if
\begin{enumerate}[label=(A)]
\item\label{ExtraCondition} $\lambda_n=0$ for all $n\in\Nt$,
\end{enumerate}
then, by the gliding hump technique we can get
\[
\sum_{j=1}^\infty \norm{x_j-y_j}<\infty.
\]

Since $(y_j)_{j=1}^\infty$ is isometrically equivalent to a sequence in $\KK$, passing to a further subsequence we can assume that there is an unconditional spreading model $\Sym_0$ of $\KK$ associated with $(y_j/\norm{y_j})_{j=1}^\infty$. Pick $\varepsilon>0$ and $\alpha=(a_k)_{k=1}^\infty$ in $c_{00}$. Let $m\in\NN$ be such that $a_n=0$ for $k>n$. There is an incresing $m$-tuple $(j_k)_{k=1}^m$ in $\NN$ such that
\[
\abs{\norm{\alpha}_\Sym- \norm{ \sum_{k=1}^m a_k x_{j_k}}}<\varepsilon,
\quad \abs{ \norm{\alpha}_{\Sym_0}- \norm{ \sum_{k=1}^m a_k \frac{y_{j_k}}{\norm{y_{j_k}} }}}<\varepsilon.
\]
Since
\begin{align*}
\norm{ \sum_{k=1}^m a_k x_{j_k}}
&\ge \norm{ \sum_{i=1}^m \sum_{k=1}^m a_k S_{\Nt_{j_i}} (x_{j_k})}\\
&\ge \norm{ \sum_{i=1}^m a_i S_{\Nt_{j_i}} (x_{j_i})} \\
&\ge c \norm{\sum_{i=1}^m a_i \frac{y_{j_i}}{\norm{y_{j_i}}}},
\end{align*}
we get
\[
\norm{\alpha}_\Sym \ge c \norm{\alpha}_{\Sym_0}-(1+c)\varepsilon.
\]
Letting $\varepsilon$ go to zero and making $\alpha$ run over all the sequences in $c_{00}$, yields $\Sym\subseteq\Sym_0$. This embedding puts an end to the proof of \ref{it:SMA}.

Let us prove \ref{it:SMB}. If \ref{ExtraCondition} holds, then $\Sym=\Sym_0$ be the principle of small perturbations.
If, on the contrary, there is $n\in\Nt$ such that $\lambda_n>0$, we apply Lemma~\ref{lem:SPDirectSum:Z} to the direct sum of $\XX_n$ and its obvious complementary space. This way we infer the existence of unconditional spreading models $\Sym_n$ and $\Sym_1$ of $\KK_n$ and $\XX$, respectively, such that either $\Sym=\Sym_n$ or $\Sym=\Sym_n\cap \Sym_1$. In any case, $\Sym\subseteq\Sym_n$.
\end{proof}

It is known \cite{AlbiacAnsorena2022PAMS}*{Theorem 4.4} that any finite direct sum built from $c_0$, $\ell_1$, $\ell_2$, $\Ts$, $\Ts^*$, $\Ts^{(2)}$ and $(\Ts^(2))^*$ has a unique unconditional basis. By Proposition~\ref{prop:SPDirectSum:A}, only $c_0$, $\ell_1$, and $\ell_2$ can be unconditional spreading models of these direct sums. By Proposition~\ref{prop:DirectSumB}, $\ell_p$ and $\ell_q$ are the unique unconditional spreading models of $Z_{p,q}$, $p$, $q\in[1,\infty]$. In particular, any unconditional spreading model of $c_0(\ell_1)$, $\ell_1(c_0)$, $c_0(\ell_2)$ or $\ell_1(\ell_2)$ is $c_0$, $\ell_1$ or $\ell_2$.

Those who still expect nice universal structural results for Banach spaces would be tempted to conjecture that the answer to Question~\ref{qt:BCLT} is positive; however, as we will see in Section~\ref{sect:main}, this is not the case.
\subsection{The hyperplane problem and related questions}
Given a Banach space $\XX$, there is $h=h(\XX)\in\NN_0$ such that
\[
\enbrace{m\in\NN \colon \XX\oplus \FF^m \simeq \XX}=\enbrace{hk\colon k\in\NN}
\]
(see \cite{AAW2022}*{Lemma 2.8}).

\begin{definition} We say that a Banach space $\XX$ is \emph{hyperplane stable} if $h(\XX)=1$, i.e., $\XX\oplus \FF \simeq \XX$. If $h(\XX)>0$, i.e., there is $h\in\FF$ such that $\XX\simeq\XX\oplus \FF^h$ we say that $\XX$ is \emph{hyperplane-sum stable}.
\end{definition}

Similarly, there is $s=s(\XX)$ such that
\[
\enbrace{m\in\NN \colon \XX^m \simeq\XX}=\enbrace{1+sk\colon k\in\NN}.
\]

\begin{definition}
A Banach space $\XX$ is \emph{square stable} if $s(\XX)=1$, i.e., $\XX^2\simeq\XX$. If $s(\XX)>0$, i.e., there is $m\in\NN\cap[2,\infty)$ such that $\XX^m \sim \XX$, we say that $\XX$ is \emph{power stable}.
\end{definition}

Most spaces that arise naturally in functional analysis are both hyperplane stable and square stable. The first counterexample to this paradigm goes back to \cite{James1964}, where James provided a Banach space $\JJ$ with $s(\JJ)=0$ and $h(\JJ)=1$.

The space $\GG$ constructed by Gowers \cite{Gowers1994hyp} witnesses that the answer to the Banach's hyperplane problem is negative, that is, there is a hyperplane nonstable Banach space. In fact, Gowers \cite{Gowers1994hyp} proved that $h(\GG)=0$ and, later on, Gowers and Maurey \cite{GowersMaurey1997} proved that $s(\GG)=0$. So, Gowers space $\GG$ witnesses the existence of Banach spaces that are neither hyperplane stable nor square stable.

Gowers went on in \cite{Gowers1994} to build a Banach space $\MM$ with $s(\MM)=2$. Subsequently, Gowers and Maurey \cite{GowersMaurey1997} constructed for each $m\in\NN$ a Banach space $\Sym_m$ with $s(\Sym_m)=m$, and another Banach space $\HH_m$ with $h(\HH_m)=m$. In the same paper, Gowers and Maurey conjectured that $h(\Sym_m)=0$ for all $m\in\NN$. To the best of our knownledge, this problem is still open. Similarly, we do not know whether the spaces $\HH_m$, $m\in\NN$, are square stable. In fact, it seems that the next question is still open.

\begin{question}
Is there a square stable Banach space failing to be hyperplane stable?
\end{question}

The Schr\"oder--Bernstein's problem for Banach spaces asks whether, assuming $\XX$ and $\YY$ are Banach spaces with $\XX\trianglelefteq \YY \trianglelefteq \XX$, then $\XX\simeq\YY$. Gowers space $\MM$ from \cite{Gowers1994} and Gowers--Maurey spaces $\HH_n$ and $\Sym_n$, $n\ge 2$, from \cite{GowersMaurey1997} proved that it is not true in general. In contrast, a Schr\"oder-Bernstein type principle for unconditional bases does hold, and its validity has been thoroughly employed in the structure theory of bases.

\begin{theorem}[\cite{Wojtowicz1988}*{Corollary 1}]\label{thm:SBUB}
Let $\XB$ and $\YB$ be unconditional bases of Banach spaces $\XX$ and $\YY$, respectively. Suppose that $\XB\subsetsim \YB$ and $\YB\subsetsim \XB$. Then $\XB\sim\YB$.
\end{theorem}

Another remarkable result related to the Schr\"oder--Bernstein's principle for unconditional bases is the \emph{power cancelling principle}, which we next record.

\begin{theorem}[\cite{AlbiacAnsorena2022b}*{Theorem 2.3}]\label{thm:SCUB}
Let $\XB$ and $\YB$ be unconditional bases of Banach spaces $\XX$ and $\YY$, respectively. Suppose that there is $m\in\NN$ such that $\XB^m\subsetsim \YB^m$. Then $\XB\subsetsim\YB$.
\end{theorem}

We emphasize that combining Theorem~\ref{thm:SBUB} with Theorem~\ref{thm:SCUB} yields Lemma~\ref{lem:AAA}. The combination of these two results also gives that $\XB\sim\YB$ when the unconditional bases $\XB$ and $\YB$ satisfy $\XB^m \sim \YB^m$ for some $m\in\NN$. The corresponding result in the category of Banach spaces is false. Indeed, the Banach spaces $\XX=\Sym_4$ and $\YY=\Sym_4^3$ are not isomorphic, while $\YY^2=\Sym_4^6 \simeq \Sym_2^2=\XX^2$. Still, it appears to be no information about the Banach space counterpart of Theorem~\ref{thm:SCUB}.

\begin{question}
Are there Banach spaces $\XX$ and $\YY$ such that $\XX^2\trianglelefteq \YY^2$ and $\XX\not\trianglelefteq \YY$?
\end{question}

The following result exhibits that the existence of a unique unconditional basis provides a connection between hyperplane-sum stability and power stability.

\begin{proposition}\label{prop:RelB}
Let $\XX$ be a Banach space with a unique unconditional basis $\XB$.

\begin{enumerate}[label=(\roman*), leftmargin=*,widest=ii]
\item Suppose $\XX$ is power stable. Then $\XB^2 \sim \XB$, and $\XX$ is both square stable and hyperplane stable.
\item Suppose $\XX$ is hyperplane-sum stable. Then $\XB\oplus\VB \sim \XB$ for any finite basis $\VB$, and $\XX$ is hyperplane stable.
\end{enumerate}
\end{proposition}

\begin{proof}
Let $\UB$ be a basis in the one-dimensional space $\FF$. Suppose there is $m\in\NN\cap[2,\infty)$ such that $\XX^m \sim \XX$. Since $\XB^m$ is a semi-normalized unconditional basis of a space isomorphic to $\XX$,
\[
\XB \subsetsim \XB \oplus \UB \subsetsim \XB^2 \subsetsim \XB^m \sim \XB.
\]
By Theorem~\ref{thm:SBUB}, $\XB \sim \XB \oplus \UB \sim \XB^2$. Consequently, $\XX \simeq \XX\oplus \FF \simeq \XX^2$.

Suppose now that there is $m\in\NN\cap[2,\infty)$ such that $\XX\oplus\FF^m \simeq \XX$. Since $\XB\oplus\UB^m$ is a semi-normalized unconditional basis of a space isomorphic to $\XX$,
\[
\XB \subsetsim \XB \oplus \UB \subsetsim \XB\oplus \UB^m \sim \XB.
\]
By Theorem~\ref{thm:SBUB}, $\XB \sim \XB \oplus \UB$ and so $\XX \simeq \XX\oplus\FF$.
\end{proof}

Proposition~\ref{prop:RelB} alerts us that if we wish to answer Question~\ref{qt:AA} we must concentrate on Banach spaces $\XX$ with $s(\XX)=0$. In fact, any square non-stable Banach space with a unique unconditional basis would provide a counterexample which would solve in the positive Question~\ref{qt:AA}.
\section{Strictly singular operators between spaces with bases}\label{sect:SSO}\noindent
Recall that an operator $T\in\Bt(\XX,\YY)$ between two Banach spaces $\XX$ and $\YY$ is \emph{strictly singular} if its restriction to any infinite-dimensional subspace of $\XX$ fails to be an isomorphic embedding. The space of all strictly singular operators from $\XX$ to $\YY$ is a vector space closed under left and right composition, and it contains the space of all compact operators (see \cite{Kato1958}).

\begin{definition} Let $\XB=(\xx_n)_{n\in\Nt}$ and $\YB=(\yy_k)_{k\in\Kt}$ be Markushevich bases of Banach spaces $\XX$ and $\YY$, respectively. Let $(\yy_n^*)_{n\in\Nt}$ be the dual basis of $\YB$. The \emph{matrix of $T\in\Bt(\XX,\YY)$ relative to $\XB$ and $\YB$} is the double sequence
\[
\mat(T;\XB,\YB)=\enpar{ \yy_k^*(T(\xx_n))}_{n\in\Nt, k\in\Kt}.
\]
An operator $T$ is \emph{roughly diagonal} relative to $\XB$ and $\YB$ if the rows and columns of $\mat(T;\XB,\YB)$ are sequences in $c_{00}(\Kt)$ and $c_{00}(\Nt)$, respectively.
\end{definition}

Note that if $T$ is a roughly diagonal operator, then $T(f)$ is finitely supported provided $f\in\XX$ is finitely supported.

Let $\XB=(\xx_n)_{n\in\Nt}$ be a Markushevich basis of a Banach space. Let $\At\subseteq\Nt$. Given an arbitrary operator $T\in \Bt([\XB|\At],\XX)$, let
\[
\mat(T;\XB)=(a_{n,k})_{n\in \At, k\in\Nt}
\]
be the matrix of $T$ relative to $(\xx_n)_{n\in\At}$ and $\XB$. The \emph{diagonal} of $T$ will be the family
\[
\dig(T;\XB)=(a_{n,n})_{n\in \At}.
\]
Since $\dig(T;\XB)\in\ell_\infty(\At)$, if $\XB$ is an unconditional basis we can safely define
\[
\diag(T;\XB)=M_{\dig(T;\XB)}[\XB]\in \Bt([\XB|\At],\XX).
\]

\begin{definition}
Let $\XX$ be a Banach space with a Schauder basis $\XB$, and $\YY$ be a Banach space. We say that $T\in\Bt(\XX,\YY)$ is \emph{strictly singular relative to $\XB$} if
\[
\lim_n T(x_n)=0
\]
for every semi-normalized block sequence $(x_n)_{n=1}^\infty$ relative to $\XB$.
\end{definition}

\begin{lemma}\label{lem:SS:A}
Let $\XX$ be a Banach space with a Schauder basis $\XB$. Suppose that $T\in\Bt(\XX)$ is strictly singular relative to $\XB$. Then $T$ is strictly singular.
\end{lemma}

\begin{proof}
We will prove the counter-reciprocal. Let $\VV$ be an infinite-dimensional subspace of $\XX$ such that $T|_\VV$ is an isomorphic embedding. Pick a uniformly separated sequence $(z_k)_{k=1}^\infty$ in the unit ball of $\VV$. Let $(\xx_n^*)_{n=1}^\infty$ be the dual basis of $\XB$. By Cantor's diagonal technique, passing to a subsequence we can assume that $(\xx_n^*(z_k))_{k=1}^\infty$ converges for all $n\in\NN$. The sequence
\[
y_k=z_{2k-1}-z_{2k}, \quad k\in\NN,
\]
as well as $(T(y_k))_{k=1}^\infty$ is semi-normalized, and $\lim_k \xx_n^*(y_k)=0$ for all $n\in\NN$. Choose $c>0$ and use the gliding hump technique to pick, passing to a subsequence, an integer sequence $(n_k)_{k=0}^\infty$ with $n_0=0$ such that the projections
\[
x_k=S_{J_k}[\XB](y_k)
\]
associated with the integer intervals $J_k=(n_{k-1},n_k]\cap \ZZ$ satisfy
\[
\norm{y_k-x_k}\le c
\]
for all $k\in\NN$. Choosing $c$ small enough, we obtain that both $(x_k)_{k=1}^\infty$ and $(T(x_k))_{k=1}^\infty$ are semi-normalized.
\end{proof}

\begin{lemma}\label{lem:SS:B}
Let $\XX$ and $\YY$ be Banach spaces with Schauder bases $\XB=(\xx_n)_{n=1}^\infty$ and $\YB=(\yy_k)_{k=1}^\infty$, respectively, and let $T\in\Bt(\XX,\YY)$. Suppose that $\XB$ is shrinking and that any roughly diagonal operator $S\in\Bt(\XX,\YY)$ with
\[
\abs{\mat(S;\XB,\YB)} \le \abs{\mat(T;\XB,\YB)}
\]
is strictly singular relative to $\XB$. Then $T$ is strictly singular relative to $\XB$.
\end{lemma}

\begin{proof}
We will prove the counter-reciprocal. Let $(\xx_n^*)_{n=1}^\infty$ and $(\yy_k^*)_{k=1}^\infty$ be the dual bases of $\XX$ and $\YY$, respectively. Let $(z_j)_{j=1}^\infty$ be a semi-normalized block sequence relative to $\XB$ such that
\[
\limsup_j \norm{T(z_j)}>0.
\]
Passing to a subsequence, we can assume that $c_1:=\inf_j \norm{T(z_j)}>0$. Set $c_2=\sup_j \norm{z_j}$. Pick $0<c<c_1$ and define $\varepsilon>0$ by
\[
2 \varepsilon c_2=c_1-c.
\]
Pick $(k_n)_{n=1}^\infty$ and in $\NN$ such that $k_n>n$ for all $n\in\NN$ and
\[
\sum_{n=1}^\infty \norm{\xx_n^*} \norm{T(\xx_n)-S_{k_n}[\YB](T(\xx_n))}\le \varepsilon.
\]
Then, there is $E_1\in\Bt(\XX,\YY)$ with $\norm{E_1}\le \varepsilon$ given by
\[
E_1(z)=\sum_{n=1}^\infty \xx_n^*(z) \enpar{T(\xx_n)-S_{k_n}(T(\xx_n)}, \quad z\in\XX.
\]
Also pick $(n_k)_{k=1}^\infty$ such that $n_k>k$ and
\[
\sum_{k=1}^\infty \norm{\yy_k} \norm{T^*(\yy_k^*) - R_{n_k}^*(\yy_k^*)}\le \varepsilon.
\]
Then, there is $E_2\in\Bt(\XX,\YY)$ with $\norm{E_2}\le \varepsilon$ given by
\[
E_2(z)= \sum_{k=1}^\infty \enpar{T^*(\yy_k^*) - R_{n_k}^*(\yy_k^*)}(z) \, \yy_k= \sum_{k=1}^\infty \yy_k^* \enpar{T\enpar{f-R_{n_k}(z)}} \, \yy_k.
\]
Let $T_1=T-E_1-E_2$ and set
\[
\mat(T;\XB,\YB)=(a_{n,k})_{n,k\in\NN}, \quad \mat(T_1;\XB,\YB)=(b_{n,k})_{n,k\in\NN}
\]
We have $b_{n,k}=0$ if $n>n_k$ or $n>n_k$, and $a_{n,k}=b_{n,k}$ otherwise. Besides,
\[
\norm{T_1(z)}\ge \norm{T(z)}- \frac{c_1-c}{c_2}\norm{z},\quad z\in\XX.
\]
Consequently, $\norm{T_1(z_j)} \ge c$ for all $j\in\NN$.
\end{proof}

\begin{lemma}\label{lem:SS:C}
Let $\XX$ and $\YY$ be Banach spaces with Schauder bases $\XB=(\xx_n)_{n=1}^\infty$ and $\YB=(\yy_k)_{k=1}^\infty$, respectively, and let $T\in\Bt(\XX,\YY)$ be a roughly diagonal operator. Suppose that $\lim_j T(z_j)=0$ whenever $(z_j)_{j=1}^\infty$ is a semi-normalized finitely supported sequence in $\XX$ such that the sequence $(C_j)_{j=1}^\infty$ defined by
\[
C_j= \bigcup_{ k\in\supp(z_j)} \{k\}\cup \supp(T(\xx_k))
\]
satisfies $\max(C_j)<\min(C_{j+1})$ for all $j\in\NN$. Then $T$ is strictly singular relative to $\XB$.
\end{lemma}

\begin{proof}
We will proceed by proving the counter-reciprocal. Define
\[
\Psi(A)= \bigcup_{n\in A} \{n\}\cup \supp(T(\xx_n)) , \quad A\subseteq\NN.
\]
Let $(y_m)_{m=1}^\infty$ be a semi-normalized block basic sequence in $\XX$ such that $\inf_m\norm{T(y_m)}>0$. We will use that for each $k\in\NN$, the set
\[
A_k:=\enbrace{m \in \NN \colon k \in \Psi( \sup(y_m))}
\]
is finite to recursively construct $(m_j)_{j=1}^\infty$ in $\NN$. Set $m_1=1$. Let $j\in\NN$ and suppose that $m_{j}$ is defined. Put
\[
k_j=\max\enpar{ \Psi( \supp(y_{m_j}))}, \quad D_j =\bigcup_{k=1}^{k_j} A_k, \mbox{ and } m_{j+1} =\min(\NN\setminus D_j).
\]
Note that $m\in D_j$ if and only if $[1,k_j]$ intersects $\Psi(\supp(y_m))$. Hence, by construction,
\[
k_j < \min\enpar{ \Psi( \supp(y_{m_{j+1}}))}.
\]
Since the sequence $(z_j)_{j=1}^\infty$ given by $z_j=y_{m_j}$ for all $j\in\NN$ is semi-normalized, and $(T(z_j))_{j=1}^\infty$ does not converge to zero, we are done.
\end{proof}

\section{Complemented subspace unconditional structure of Banach spaces with a small space of operators}\label{sect:D+S}\noindent
In 1994, Gowers constructed an example of a Banach space $\GG$ with an unconditional basis which is not isomorphic to $\GG\oplus \RR$, thus giving a negative answer to a question of Banach known as the hyperplane problem. Subsequently, Gowers and Maurey \cite{GowersMaurey1997} proved the stronger result that every operator defined on $\GG$ is the sum of a diagonal operator and a strictly singular one. One could say then that $\GG$ has as few operators as possible. This is the motivation to introduce a new kind of bases, which we call strictly singular plus diagonal bases, and study the complemented subspace unconditional structure of those Banach spaces with unconditional bases which additionally have this property.

We will denote by $\Ft(\XX,\YY)$ the set of all \emph{semi-Fredholm operators} from $\XX$ to $\YY$, that is, all $T\in\Bt(\XX,\YY)$ such that
\begin{enumerate}[label=(\alph*),leftmargin=*,widest=c]
\item\label{F:Range} $T(\XX)$ is closed in $\YY$, and
\item $d_1:=\dim(\Ker(T))<\infty$.
\end{enumerate}
If
\[
d_2:=\codim(f(\XX);\YY)<\infty,
\]
condition~\ref{F:Range} is redundant. The Fredholm index of $T$ is the number
\[
\ind(T)=d_1-d_2.
\]
If $\ind(T)\in\ZZ$, that is $d_2<\infty$, $T$ is a \emph{Fredholm operator}. The existence of such an operator implies a close relation between $\XX$ to $\YY$. In fact, if $\ind(T)\ge 0$, then $\XX\simeq \YY\oplus \FF^{\ind(T)}$. In turn, if $\ind(T)\le 0$, then $\XX\simeq\YY_0$, where $\YY_0$ is an arbitrary subspace of $\YY$ with $\codim(\YY_0,\YY)=-\ind(T)$.

We will use the following perturbation result.

\begin{theorem}[see \cite{GowersMaurey1997}*{Lemma 18}]\label{cor:Fred+SS}
Let $\XX$ and $\YY$ be Banach spaces, $T\in\Bt(\XX,\YY)$ be a semi-Fredholm operator and $S\in\Bt(\XX,\YY)$ be strictly singular. Then $T+S$ is a semi-Fredholm operator with $\ind(T)=\ind(T+S)$.
\end{theorem}

\begin{definition} A basis $\XB$ of a Banach space $\XX$ has the \emph{diagonal plus strictly singular property} (D+S for short) if for all $T\in\Bt(\XX)$ the difference operator $T-\diag(T;\XB)$ is strictly singular.
\end{definition}

\begin{lemma}\label{lem:nulldiag}
Let $\XB=(\xx_n)_{n\in\Nt}$ be a D+S unconditional basis of a Banach space. Let $\At\subseteq\Nt$ and $T\in\Bt([\XB|\At],\XX)$ be such that $\dig(T;\XB)\in c_0(\At)$. Then $T$ is strictly singular.
\end{lemma}

\begin{proof}
Since $M_\beta[\XB]$ is compact if $\beta\in c_0(\At)$, so is $\diag(T;\XB)$. Let $S=T\circ M_{\At}[\XB]$. By assumption, $R:=S - \diag(S;\XB)$ is strictly singular. Hence, the operator $Q:=R|_{[\XB|A]}$ is strictly singular. Therefore, $P:=Q+\diag(T;\XB)$ is strictly singular. Since $P=T$, we are done.
\end{proof}

A partial converse of Lemma~\ref{lem:nulldiag} also holds.

\begin{lemma}\label{lem:GMIdea}
Let $\XB=(\xx_n)_{n\in\Nt}$ be a D+S unconditional basis of a Banach space $\XX$. Let $\At\subseteq\Nt$ and $T\colon [\XB|\At] \to \XX$ be an isomorphic embedding.Then $\liminf\dig(T;\XB)>0$.
\end{lemma}

\begin{proof}
Put $\dig(T;\XB)=(d_n)_{n\in\At}$ and asume by contradiction that
\[
\liminf_{n\in\At} d_n=0.
\]
Then, there is $\Mt\subseteq\At$ infinite such that $\lim_{n\in\Mt} d_n=0$. Set $S=T|_{ [\XB|\Mt]}$. On the one hand, $S$ is an isomorphic embedding. On the other hand, $S$ is strictly singular by Lemma~\ref{lem:nulldiag}.
\end{proof}

We will advance our understanding of D+S unconditional bases $\XB$ by proving that any two subbases of $\XB$ are permutatively equivalent only in the trivial case.

\begin{lemma}\label{lem:DSSTrivalEquivalence}
Let $\XB=(\xx_n)_{n\in\Nt}$ be a semi-normalized D+S unconditional basis. If $\Mt$ and $\Kt$ are subsets of $\Nt$ with $(\xx_n)_{n\in\Mt} \sim (\xx_n)_{n\in\Kt}$, then
\[
\abs{\Mt\setminus\Nt}=\abs{\Nt\setminus\Mt}<\infty.
\]
\end{lemma}

\begin{proof}
Let $\pi\colon\Mt\to\Kt$ be a bijection for which there is an isomorphism $S$ from $[\XB|\Mt]$ onto $[\XB|\Kt]$ such that $S(\xx_n)=\xx_{\pi(n)}$ for all $n\in\Mt$. Note that
\[
\dig(S;\XB)=(\delta_{\pi(n),n})_{n\in\Mt}.
\]
By Lemma~\ref{lem:GMIdea}, there is $F\subseteq\Mt$ finite such that $\pi(n)=n$ for all $n\in\Mt\setminus F$.
\end{proof}

Gowers and Maurey derived the fact that Gowers space $\GG$ is neither power stable nor hyperplane stable from the following result, which we rewrite with our ingredients.

\begin{theorem}[c.f.\@ \cite{GowersMaurey1997}*{Corollary 30}]\label{thm:GMHS}
Let $\XX$ be a Banach space with a D+S unconditional basis. Then $\XX$ is not isomorphic to any of its proper subspaces.
\end{theorem}

\begin{proof}
Let $T\colon \XX\to \XX$ be an isomorphic embedding onto a subspace $\YY$. By Lemma~\ref{lem:GMIdea}, $\diag(T;\XB)$ is a Fredholm operator with null index. Hence, by Theorem~\ref{cor:Fred+SS}, $T$ is a Fredholm operator with null index. Therefore, $\YY=\XX$.
\end{proof}

The following result reveals that the existence of a D+S unconditional basis leads to a homogeneous complemented subspace structure. We note that the matrix spaces $Z_{p,q}$ with a unique unconditional basis have a similar homogeneous complemented structure, as was proved by Bourgain et al.\@ in \cite{BCLT1985}.

\begin{theorem}\label{thm:AAC}
Let $\XX$ be a Banach space with a D+S unconditional basis $\XB=(\xx_n)_{n\in\Nt}$. Then any complemented subspace of $\XX$ is congruent to $[\XB|\At]$ for some $\At\subseteq\Nt$.
\end{theorem}

\begin{proof}
Let $P_1$, $P_2\in\Bt(\XX)$ be complementary projections. Set $\YY_j=P(\XX_j)$, and $\dig(P_j;\XB)=(a_{n,j})_{n\in\Nt}$
$j=1$, $2$. Assume by contradiction that
\[
\gamma:=\limsup_{n\in\Nt} \min_{j=1,2} \abs{a_{n,j}}>0.
\]
Then there is $\Mt\subseteq\Nt$ infinite such that the families $(a_{n,j})_{n\in \Mt}$, $j=1$, $2$, are bounded away from zero. By the Bolzano--Weierstrass theorem, there are $\Kt\subseteq \Mt$ infinite and $\alpha_1$, $\alpha_2\in(0,\infty)$ such that $\lim_{n\in \Kt} a_{n,j}=\alpha_j$, $j=1$, $2$. Set
\[
T=\frac{P_1}{\alpha_2}-\frac{P_2}{\alpha_1}, \quad S=T|_{[\XB|\Kt]}.
\]
On one hand, $T$ is an automorphism of $\XX$. On the other hand $\dig(S;\XB)\in c_0(\Kt)$, whence $S$ is strictly singular by Lemma~\ref{lem:nulldiag}.

This absurdity proves that $\gamma=0$. Hence, there is a partition $(\Nt_1,\Nt_2)$ of $\Nt$ such that $\lim_{n\in\Nt_j} a_{n,i}=0$, $\{i,j\}=\{1,2\}$.
Set
\[
T_j=M_{\Nt_j}[\XB] \circ \diag(P_j;\XB), \quad R_j=T_j\circ P_j, \quad j=1,2.
\]
By Lemma~\ref{lem:nulldiag}, $P_j-T_j$, $j=1$, $2$, are strictly singular operators. Consequently, the operator
\[
\Id_\XX-R_1-R_2=P_1^2+P_2^2-T_1\circ P_1 -T_2\circ P_2=(P_1 - T_1)\circ P_1 + (P_2 - T_2)\circ P_2
\]
is strictly singular. By Theorem~\ref{cor:Fred+SS}, $R_1+R_2$ is a Fredholm operator with null index. Therefore, the operator
\[
(R_1,R_2)\colon \YY_1\oplus \YY_2\to [\XB|\Nt_1] \oplus [\XB|\Nt_2]
\]
is a Fredholm operator with null index. Consequently, $R_j\colon \YY_j \to [\XB|\Nt_j]$, $j=1$, $2$, are Fredholm operators with opposite indices. Assume without loss of generality $k:=\ind(R_1)\ge 0$. Then,
\begin{multline*}
\abs{\Nt_2}\ge\codim(R_2(\YY_2); [\XB|\Nt_2])) \\ \ge \codim(R_2(\YY_2); [\XB|\Nt_2])) - \dim(\Ker(R_2))=k.
\end{multline*}
Pick $F\subseteq \Nt_2$ with $\abs{F}=k$ and set $\Mt_1=\Nt_1\cup F$ and $\Mt_2=\Nt_2\setminus F$. We infer that $\YY_j\simeq [\XB|\Mt_j]$, $j=1$, $2$.
\end{proof}

We say that a family $\XB=(\xx_k)_{k\in\Kt}$ in a Banach space $\XX$ is \emph{complemented} if $[\XB]$ is a complemented subspace of $\XX$. If $\XB$ is unconditional and complemented, then there is $\XB^*=(\xx_k^*)_{k\in\Kt}$ in $\XX^*$ such that the mapping
\[
f\mapsto Q(f):= \sum_{k\in\Kt} \yy_k^*(f) \, \yy_k, \quad f\in\XX,
\]
defines a bounded linear projection from $\XX$ onto $[\XB]$. If this is the case we say that $\XB^*$ are \emph{projecting functionals} for $\XB$, and that $Q=Q[\XB,\XB^*]$ is the projection associated with the biorthogonal system $(\XB,\XB^*)$.

Given a K\"othe space $\KK$ and a family $\YB=(\yy_k)_{k\in\Kt}$ in $\KK\setminus\{0\}$, we denote by $\LL[\YB]$ the atomic lattice over $\Kt$ defined by the function norm
\[
(a_k)_{k\in \Kt} \mapsto \norm{ \enpar{ \sum_{k\in\Kt} \abs{a_n}^2 \abs{\yy_k}^2}^{1/2}}_{\KK}.
\]
This construction goes back to \cite{BCLT1985}, where it is called the latification procedure.

\begin{lemma}[\cite{AlbiacAnsorena2025}*{Theorem 4.2}]\label{lem:lattification}
Let $\KK$ be a a K\"othe space. Let $(\xx_k)_{k\in\Kt}$ be complemented unconditional family in $\KK$ with projecting functionals $(\xx_k^*)_{k\in\Kt}$. Let $\YB=(\yy_k)_{k\in\Kt}$ be a family in $\KK$ such that
\[
\inf_{k\in\Kt} \abs{ \xx_k^*(\yy_k)}>0
\]
and there is a constant $C$ such that $\abs{\yy_k} \le C \abs{\xx_k}$ for all $k\in\Kt$. Then $\YB$ is equivalent to the unit vector system of the atomic lattice $\LL[\YB]$.
\end{lemma}

The following lemma contains a substantial part of our proof of Theorem~\ref{thm:CKBanach}.

\begin{lemma}\label{lem:ComD+SS}
Let $\XB=(\xx_k)_{k\in \Kt}$ be a complemented unconditional basic sequence of a Banach space $\UU$ with a D+S unconditional basis $\UB$. If $\XB$ and $\UB$ are semi-normalized, then $\XB \subsetsim \UB$.
\end{lemma}

\begin{proof}
We can replace $\UU$ with a minimal normalized atomic lattice $\KK$, and $\UB$ with its unit vector system $(\ee_n)_{n\in\Nt}$. Let $(\xx_k^*)_{k\in\Kt}$ be projecting functionals for $\XB$. Consider the matrices
\[
a_{k,n}=\xx_k^*(\ee_n), \, b_{k,n} = \ee_n^*(\xx_k),\, c_{k,n}=a_{k,n}b_{k,n} \quad k \in\Kt, \, n\in \Nt,
\]
and its rows $\alpha_k=(a_{k,n})_{n\in\Nt}$, $\beta_k=(a_{k,n})_{n\in\Nt}$ and $\gamma_k=(c_{k,n})_{n\in\Nt}$, $k\in\Kt$.
Since $(\alpha_k)_{k\in\Kt}$ is a bounded family in $\ell_\infty(\Nt)$ and $(\beta_k)_{k\in\Kt}$, is a bounded family in $c_0(\Nt)$, $(\gamma_k)_{k\in\Kt}$ is a bounded family in $c_0(\Nt)$. Set
\[
C_k=\norm{\gamma_k}_\infty=\max_{n\in\Nt} \abs{c_{k,n}}, \quad k\in\Kt.
\]
Assume by contradiction that $\liminf_{k\in \Kt} C_k=0$. Then, there is $\Jt\subseteq \Nt$ infinite such that $\sum_{k\in\Jt} C_k<\infty$. Set $\YB=(\xx_k)_{k\in\Jt}$ and $\YB^*=(\xx_k^*)_{k\in\Jt}$. The diagonal $\dig(Q;\XB)=(d_n)_{n\in\Nt}$ of the projection $Q=Q[\YB,\YB^*]$ is given by
\[
d_n= \sum_{k\in\Jt} c_{k,n}, \quad n\in\Nt.
\]
By the dominated converge theorem, $\lim_{n\in\Nt} d_n=0$. Therefore, by Lemma~\ref{lem:nulldiag}, $Q$ is strictly singular. Since $Q$ is the indentity on $[\YB]$, we reach an absurdity.

This contradiction proves that the families $\gamma_k$, $k\in\Kt$, eventually peak, that is, there is a partition $(\Kt_1,\Kt_2)$ of $\Kt$ such that $\abs{\Kt_1}<\infty$ and for each $k\in\Kt_2$ there is $n(k)\in\Nt$ such that
\[
\inf_{k\in\Kt_2} \abs{c_{k,n(k)}}>0.
\]
Consequently,
\[
\inf_{k\in\Kt_2} \abs{a_{k,n(k)}}>0, \quad \inf_{k\in\Kt_2} \abs{b_{k,n(k)}}>0.
\]
Set $\UB=(\ee_{n(k)})_{k\in\Kt_2}$ and
\[
\At_n=\enbrace{k\in\Kt_2 \colon n(k)=n}, \quad n\in\Nt.
\]
By Lemma~\ref{lem:lattification}, $(\xx_n)_{n\in\Kt_2}$ is equivalent to the unit vector system of $\LL[\UB]$. In turn, $\LL[\UB]$ is naturally lattice isometric to
\[
\TT:=\enpar{\bigoplus_{n\in\Nt} \ell_2(\At_n)}_{\KK}.
\]

Since $[\XB | \Kt_1] \oplus \TT$ is isomorphic to $\KK$, there is an isomorphic embedding $S\colon \TT\to \KK$.
For each $n\in\Nt$, let $L_n\colon \ell_2(A_n)\to \TT$, $n\in\Nt$, be the canonical embedding, and set
\[
\varphi_n=\ee_{n}^* \circ S \circ L_n.
\]
Since $\varphi_n$ is a functional on $\ell_2(\At_n)$, $\Ker(\varphi_n)$ is nonnull as long as
\[
n\in \Nt_1:=\enbrace{n \in \Nt \colon \abs{\At_n}\ge 2}.
\]
Assume by contradiction that $\Nt_1$ is infinite. Pick for each $n\in\Nt_1$ $f_n\in \ell_2(\At_n)$ such that $\varphi_n(f_n)=0$ and $\norm{f_n}=1$. There is an isometric embedding
\[
R \colon \KK[\Nt_1] \to \TT
\]
such that $R(\ee_n)=f_n$ for all $n\in\Nt_1$. On the one hand $S\circ R$ is an isomorphic embedding. On the other hand, by construction, $\diag(S\circ R)=0$. Consequenty, $S\circ R$ is strictly singular by Lemma~\ref{lem:nulldiag}.

This absurdity shows that $\Nt_1$ is finite. Hence, if
\[
\Nt_2=\enbrace{n \in \Nt \colon \abs{\At_n}=1}, \quad
\At=\sqcup_{n\in\Nt_1} \At_n,
\]
the unit vector system of $\TT$ is permutatively equivalent to the unit vector system of $\ell_2(\At)\oplus \KK[\Nt_2]$. Consequently, if $\Dt=\At\sqcup \Kt_1$, $\XB$ is permutatively equivalent to the unit vector system of
$\ell_2(\Dt)\oplus \KK[\Nt_2]$.

Assume by contradiction that $\Dt$ is infinite. Then, by Theorem~\ref{thm:AAC}, there is $\St\subseteq\Nt$ such that $\KK[\St] \simeq\ell_2$. Since, $\ell_2$ has a unique unconditional basis, $(\ee_n)_{n\in\St}$ is permutatively equivalent to the unit vector system of $\ell_2(\St)$. In particular, $(\ee_n)_{n\in\St}\sim (\ee_n)_{n\in\Ht}$ for every $\Ht\subseteq\St$ infinite. Since this assertion contradicts Lemma~\ref{lem:DSSTrivalEquivalence}, $\Dt$ is finite. Set
\[
k=\abs{\Nt \setminus\Nt_2}-\abs{\Dt} \in\ZZ\cup\{\infty\}.
\]
If $k<0$, then there would be $\Mt\subseteq \Kt$ with $\abs{\Kt\setminus \Mt}=-k$ such that $(\xx_n)_{n\in\Mt}$ would be equivalent to the unit vector system of $\KK$. Hence, $\KK$ there would be isomorphic to $[\XB|\Mt]$. Since this assertion contradicts Theorem~\ref{thm:GMHS}, $k\ge 0$, whence there is $\Nt_2\subseteq\Nt_3\subseteq\Nt$ such that $\abs{\Nt_3\setminus \Nt_2}=k$. We infer that $\XB$ is permutatively equivalent to the unit vector system of $\KK[\Nt_3]$.\end{proof}

\begin{theorem}\label{thm:Gathered}
Let $\XX$ be a Banach space with a D+S unconditional basis $\XB$. Then every complemented subspace $\YY$ of $\XX$ has a unique unconditional basis $\YB$. Besides, if $\XB$ and $\YB$ are semi-normalized, then $\YB\subsetsim\XB$.
\end{theorem}

\begin{proof}
Property D+S passes to subbases by Lemma~\ref{lem:nulldiag}. So, by Theorem~\ref{thm:AAC}, it suffices to prove that $\XX$ has a unique unconditional basis. Let $\YB$ be another seminormalized unconditional basis of $\XX$. By Lemma~\ref{lem:ComD+SS}, there is a subbasis $\UB$ of $\XB$ such that $\YB\sim \UB$. Hence, $\XX\simeq [\UB]$. By Theorem~\ref{thm:GMHS}, $\UB=\XB$.
\end{proof}

The following result, combined with Theorem~\ref{thm:AAC}, settles the structure of complemented subspaces of Banach spaces with a D+S unconditional basis.

\begin{theorem}\label{thm:SCS}
Let $\XB=(\xx_n)_{n\in\Nt}$ be a D+S semi-normalized unconditional basis of a Banach space. Suppose that $\Mt$ and $\Kt$ are subsets of $\Nt$ with $[\XB \mid \Mt] \simeq [\XB \mid \Kt]$. Then
\[
\abs{\Mt\setminus\Nt}=\abs{\Nt\setminus\Mt}<\infty.
\]
\end{theorem}

\begin{proof}
Just combine Theorem~\ref{thm:Gathered} with Lemma~\ref{lem:DSSTrivalEquivalence}.
\end{proof}

The following consequence of Theorem~\ref{thm:SCS} tells us that complemented subspaces of Banach spaces with a D+S unconditional basis exhibit no self-similarity.

\begin{corollary}\label{cor:HSNotEmbed}
Let $\XX$ be a Banach space with a D+S unconditional basis $\XB=(\xx_n)_{n\in\Nt}$. Let $\YY$ be a Banach space which is isomorphic to some of its proper complemented subspaces. Then $\YY$ does not complementably embed into $\XX$. In particular, if $\YY$ is hyperplane-sum stable or power stable, it does not complementably embed into $\XX$.
\end{corollary}

\begin{proof}
Assume that $\XB$ is semi-normalized. Let $\UU$ be a proper complemented subspace of $\UU$. Suppose that $\YY \trianglelefteq \XX$. By Theorem~\ref{thm:AAC}, there are $\Bt\subsetneq\At\subseteq\Nt$ such that $\YY\simeq [\XB|\At]$ and $\UU\simeq [\XB | \Bt ]$. By Theorem~\ref{thm:SCS}, $\YY$ and $\UU$ are not isomorphic.
\end{proof}

Since unconditional bases of reflexive Banach spaces are shrinking, the following result shows the relevance in assuming the basis to be shrinking in Lemma~\ref{lem:SS:B}. This condition seems to be missing from the discussion starting \cite{GowersMaurey1997}*{Section 5.1}.

\begin{corollary}
Let $\XX$ be a Banach space with a D+S unconditional basis $\XB$. Then $\XX$ is reflexive.
\end{corollary}

\begin{proof}
By Corollary~\ref{cor:HSNotEmbed}, neither $\ell_1$ nor $c_0$ complementably embeds into $\XX$. It is known \cite{James1950} that this implies $\XX$ to be reflexive.
\end{proof}
\section{Uniqueness of unconditional basis of Gowers space and its convexifications}\label{sect:main}\noindent
The Banach space constructed by Gowers in \cite{Gowers1994hyp} to solve Banach's hyperplane problem is an atomic lattice. In this section we solve in the positive Questions~\ref{qt:AA}, \ref{qt:UTAPCloselp}, \ref{qt:lpFC}, and \ref{qt:BCLT} by showing that Gowers space $\GG$ and its $p$-convexified spaces $\GG^{(p)}$, $1\le p<\infty$, have a unique unconditional basis.

For the reader's ease we introduce the specific terminology and properties related to $\GG$ we will use to address this task.

Recall that a function $f\colon[1,\infty)\to[1,\infty)$ belongs to he class $\Ft$ introduced by Schlumprecht \cite{Schlumprecht1991} if
\begin{itemize}
\item $f(1)=1$ and $f(x)<x$ for for $x\in(1,\infty)$,
\item $\lim_{x\to\infty} f(x)=\infty$ and $\lim_{x\to\infty} x^{-q} f(x)=0$ for all $q>0$,
\item $f(xy)\le f(x) f(y)$ for all $x$, $y\in[1,\infty)$, and
\item the function $x\mapsto x/f(x)$ is concave.
\end{itemize}

Let $\GG[f]$ be the Banach space constructed in the same way as in \cite{Gowers1994hyp} from a given function $f\in\Ft$.
Gowers space $\GG$ is just $\GG[f]$ in the case when $f(x)=\log_2(1+x)$ for $x\ge 1$.
In general, $\GG[f]$ is a minimal normalized K\"othe space over $\NN$ for any $f\in\Ft$. The mere definition of $\GG[f]$ implies that
\begin{equation*}
\sum_{j=1}^n \norm{y_j} \le f(n) \norm{\sum_{j=1}^n y_j }
\end{equation*}
for all block sequences $(y_j)_{j=1}^n$ relative to the unit vector system $\Et=(\ee_n)_{n=1}^\infty$ of $\GG[f]$. This lower estimate prevents $\ell_q$, $1<q\le \infty$, from being crudely block finitely representable in $\Et$. In turn, by \cite{Gowers1994hyp}*{Lemma~7}, for any $\varepsilon>0$, any $m\in\NN$, and any infinite normalized block basic sequence relative to the canonical basis of $\GG[f]$ there are $n\ge m$ a further normalized block basic sequence $(y_j)_{j=1}^n$ such that
\[
\norm{ \sum_{j=1}^n y_j } \le (1+\varepsilon) \frac{n}{f(n)}.
\]
Therefore, no block basic sequence of the canonical basis $\GG[f]$ is equivalent to the canonical basis of $\ell_1$.
Let us record some ready consequences of these facts.
\begin{proposition}\label{prop:KrivineClaim}
Let $f\in\Ft$.
\begin{enumerate}[label=(\roman*), leftmargin=*, widest=iii]
\item\label{it:sue} If $\KK$ is the restriction of $\GG[f]$ to an infinite subset of $\NN$, then $\sue(\KK)=1$.
\item\label{it:embeds} Neither the canonical basis of $c_0$ nor $\ell_q$, $1\le q<\infty$, are equivalent to a block basic sequence relative to the canonical basis of $\GG[f]$.
\item\label{it:l1block} $\ell_1$ is finitely block representable in any (infinite) block basic sequence relative to the canonical basis of $\GG[f]$ (cf.\@ \cite{Gowers1994hyp}*{Lemma 1}).
\end{enumerate}
\end{proposition}

\begin{proof}
\ref{it:sue} and \ref{it:embeds} are clear, and \ref{it:l1block} is a consequence of Krivine's Theorem (see \cite{AlbiacKalton2016}*{Theorem 12.3.9}).
\end{proof}

Given $1\le p<\infty$ and $f\in\Ft$, let $\GG[p,f]$ denote the $p$-convexification of $\GG[f]$. The following theorem gathers some properties of these spaces. Before stating and proving them, we note that those features of Banach lattices that only involve pairwise disjointly supported families plainly pass to their convexifications. In particular,
\begin{equation}\label{eq:GMSLE}
\frac{1}{ f^{1/p}(m)} \enpar{ \sum_{j=1}^m \norm{y_j}^p}^{1/p} \le \norm{\sum_{j=1}^m y_j} \le \enpar{ \sum_{j=1}^m \norm{y_j}^p}^{1/p}.
\end{equation}
for all block sequences $(y_j)_{j=1}^m$ relative to the unit vector system of $\GG[p,f]$.

\begin{theorem}\label{thm:GatherGp}
Let $f\in\Ft$ and $1\le p<\infty$.
\begin{enumerate}[label=(\roman*), leftmargin=*, widest=iii]
\item\label{it:suep} If $\KK$ is the restriction of $\GG[p,f]$ to an infinite subset of $\NN$, then $\sue(\KK)=p$.
\item\label{it:embedsp} Neither $\ell_q$, $1\le q<\infty$, nor $c_0$ embed into $\GG[p,f]$.
\item\label{it:ShrBC} The canonical basis of $\GG[p,f]$ is a boundedly complete and shrinking unconditional basis.
\item\label{it:GpReflexive} $\GG[p,f]$ is reflexive.
\item\label{it:GpSM} If $(\Sym,\norm{\cdot}_\Sym)$ is an unconditional spreading model of $\GG[p,f]$, then
\[
\frac{1}{f^{1/p}(m)} \enpar{\sum_{j=1}^m \abs{a_j}^p}^{1/p} \le \norm{s}_{\Sym} \le \enpar{\sum_{j=1}^m \abs{a_j}^p}^{1/p}.
\]
for all $s=\sum_{j=1}^m a_j\, \ee_j\in c_{00}$.
\item\label{it:lpCompBlock} $\ell_p$ is complementably finitely block representable in any (infinite) block basic sequence relative to the canonical basis of $\GG[p,f]$.
\end{enumerate}
\end{theorem}

\begin{proof}
\ref{it:suep} follows from Proposition~\ref{prop:KrivineClaim}\ref{it:sue}. To obtain \ref{it:embedsp}, just combine Proposition~\ref{prop:KrivineClaim}\ref{it:embeds} with \cite{AADK2021}*{Proposition 2.3}. By James theorems from \cite{James1950} (see \cite{AlbiacKalton2016}*{Theorems 3.3.1 and 3.3.2}), \ref{it:embedsp} implies \ref{it:ShrBC} and \ref{it:GpReflexive}. Combining Lemma~\ref{lem:SMBBS} with \eqref{eq:GMSLE} yields \ref{it:GpSM}.

Note that Theorem~\ref{thm:Complp} gives that $\ell_p$ is crudely disjointly finitely complementably representable in $\GG[p,f]$. To prove the stronger result \ref{it:lpCompBlock}, we consider the conjugate exponent $q$ of $p$. By \eqref{eq:GMSLE} and \cite{LinTza1979}*{Proposition 1.f.5}, the norm of the dual space of $\GG[p,f]$ satisfies the estimates
\[
\enpar{\sum_{k=1}^m \norm{x_k^*}^q}^{1/q} \le \norm{\sum_{k=1}^m x_k^*} \le f^{1/p}(m) \enpar{\sum_{k=1}^m \norm{x_k^*}^q}^{1/q}
\]
for every block basic sequence $(x_k^*)_{k=1}^m$ relative to the canonical basis. These estimates prevent $\ell_r$, $r\not=q$, from being crudely finitely block representable in the canonical basis. Consequently, combining Krivine's theorem with Lemma~\ref{lem:KrivineA} we are done.
\end{proof}

The value of $\ile\enpar{\GG[p,f]}$ is unclear to the authors. The subtle difference between block representability and disjoint representability should be clarified to address this question. We also point out that Theorem~\ref{thm:GatherGp} gives valuable information about the spreading-model structure of $\GG[p,f]$, but does not completely settle the issue.

Gowers and Maurey proved that the canonical unconditional basis of $\GG$ has the D+S property. Here, we will prove a more general result. First, let us record a technical result from \cite{GowersMaurey1997} that we will use, followed by a combinatorial lemma.
\begin{lemma}[see \cite{GowersMaurey1997}*{Proof of Lemma 27}]\label{lem:GMInside}
Let $f\in\Ft$. Let $(x_n)_{n=1}^\infty$ and $(y_n)_{n=1}^\infty$ be sequences of nonnull vectors in $c_{00}$, which we regard within $\GG[f]$. Set $A_n=\supp(x_n)$, $B_n=\supp(y_n)$ and $C_n=A_n\cup B_n$ for all $n\in\NN$. Suppose that
\begin{itemize}[leftmargin=*]
\item $A_n\cap B_n=\emptyset$ for all $n\in\NN$,
\item $\max(C_n)<\min(C_{n+1})$ fot all $n\in\NN$, and
\item There is a bounded linear map $V\colon [x_n \colon n\in\NN] \to \GG[f]$ such that $V(x_n)=y_n$ for all $n\in\NN$.
\end{itemize}
Then, $\lim_n \norm{y_n}/\norm{x_n}=0$.
\end{lemma}

\begin{lemma}\label{lem:average}
Let $\XX$ be a Banach space with a Markushevich basis $\XB=(\xx_n)_{n\in\Nt}$. Let $T\in \Bt(\XX)$ be such that $\diag(T;\XB)=0$. Let $x\in\XX$ and $J\subseteq\Nt$ finite be such that
\[
\bigcup_{n\in\supp(f)} \{n\} \cup \supp(T(\xx_n))\subseteq J.
\]
Then
\[
T(x)= 2 \Ave_{A\subseteq J} S_A[\XB] \circ T \circ S_{J\setminus A}[\XB](x).
\]
\end{lemma}

\begin{proof}
It suffices to prove the result in the case when $x=\xx_n$, where $\{n\} \cup \supp(T(\xx_n))\subseteq J$. Set $\mat(T)=(a_{n,k})_{(n,k)\in\Nt^2}$.
We have
\[
\sum_{A\subseteq J} S_A[\XB] \circ T \circ S_{J\setminus A}(\xx_n)
=\sum_{ A\subseteq J\setminus\{n\} } \sum_{k\in A} a_{n,k} \, \xx_k
=2^{\abs{J}-1} \sum_{k \in J\setminus\{n\}} a_{n,k} \, \xx_k.
\]
Since $\supp(\xx_n) \subseteq J\setminus\{n\}$, we are done.
\end{proof}

\begin{theorem}\label{thm:GMD+SS}
Let $f\in\Ft$ and $1\le p<\infty$. Then the canonical unconditional basis of $\GG[p,f]$ has the D+S property.
\end{theorem}

\begin{proof}
Assume by contradiction that the canonical basis $\Et=(\ee_n)_{n=1}^\infty$ of $\XX=\GG[p,f]$ fails to have the D+S property. Then there is a non-strictly singular operator $T_1\in\Bt(\XX)$ with $\diag(T_1;\Et)=0$. By Lemmas~\ref{lem:SS:A}, ~\ref{lem:SS:B}, and~\ref{lem:SS:C}, there are $c>0$, $T_2\in\Bt(\XX)$ with $\diag(T_2;\Et)=0$, an increasing integer sequence $(n_j)_{j=0}^\infty$ with $n_0=0$, and a semi-normalized sequence $(x_j)_{j=1}^\infty$ in $\XX$ such that $\norm{T_2(x_j)}\ge 2c$ and
\[
\bigcup_{n\in \supp(x_j)} \{n\} \cup \supp(T_2(\ee_n)) \subseteq J_j:=(n_{j-1},n_j]\cap\ZZ
\]
for all $j\in\NN$. By Lemma~\ref{lem:average}, there is a partition $(A_j,B_j)$ of $J_j$ such that the vector
\[
y_j=S_{A_j}[\Et] \circ T \circ S_{B_j}[\Et](x_j)
\]
satisfies $\norm{y_j}\ge c$. Set $A=\cup_{j=1}^\infty A_j$, $B=\cup_{j=1}^\infty B_j$ and $T_3=S_A\circ T_2 \circ S_B$. Put
\[
z_j=S_{B_j}[\Et] (x_j),
\]
$C_j=\supp(z_j)$ and $D_j=\supp(y_j)$ for all $j\in\NN$. We have $T_3(z_j)=y_j$, $C_j\cap D_j=\emptyset$, and $C_j\cup D_j\subseteq J_j$ for all $j\in\NN$.

For each $j\in\NN$, we define vectors in $\YY:=\GG[f]$ by $u_j=\abs{z_j}^{1/p}$ and $v_j=\abs{y_j}^{1/p}$. The sequences $\UB:=(u_j)_{j=1}^\infty$ and $\VB:=(v_j)_{j=1}^\infty$ are semi-normalized, and $\supp(u_k)=C_k$ and $\supp(v_k)=D_j$ for all $j\in\NN$. Since the vectors of $\UB$ are pairwise disjointly supported, there is a bounded linear map $T_4\colon [\UB] \to \YY$ such that $T(u_j)=v_j$ for all $j\in\NN$. By Lemma~\ref{lem:GMInside}, $\lim_j \norm{v_j}/\norm{u_j}=0$. This absurdity puts an end to the proof.
\end{proof}

We are now in a position to provide positive answers to Questions~\ref{qt:AA}, \ref{qt:UTAPCloselp}, \ref{qt:lpFC}, and a negative answer to Question~\ref{qt:BCLT}.

\begin{theorem}\label{thm:CKBanach}
Suppose $\XX$ is an infinite-dimensional complemented subspace of $\GG[p,f]$, where $f\in\Ft$ and $1\le p<\infty$. Then:
\begin{enumerate}[label=(\roman*), leftmargin=*, widest=iii]
\item $\XX$ has a unique unconditional basis.
\item $\XX$ fails to be power stable.
\item $\XX$ fails to be hyperplane-sum stable.
\item $\sue(\XX)=p$.
\item Neither $c_0$ nor $\ell_q$, $1\le q<\infty$, $q\not=p$, are spreading models for $\XX$. In particular, if $p\notin\{1,2\}$ then $\GG[p,f]$ has a spreading other than $c_0$, $\ell_1$, or $\ell_2$.
\end{enumerate}
\end{theorem}

\begin{proof}[Proof of Theorem~\ref{thm:CKBanach}]
Just combine Theorem~\ref{thm:GMD+SS}, Theorem~\ref{thm:Gathered}, Theorem~\ref{thm:GMHS}, and Theorem~\ref{thm:GatherGp}.
\end{proof}
\section{Open Questions}\label{sect:open}\noindent
A well-known general problem in the study of uniqueness of lattice structure is the following.
\begin{question}[see \cite{AAW2022}*{Question 1.1}]\label{qt:UTAPX+Y}
Let $\XX$ and $\YY$ Banach spaces with a unique unconditional basis. Does $\XX\oplus \YY$ have a unique unconditional basis?
\end{question}

Since the theory of uniqueness of unconditional basis has been mainly developed for Banach spaces isomorphic to their square, addressing Question~\ref{qt:UTAPX+Y} in the case when $\XX=\YY$ has never been on the table. In light of Theorem~\ref{thm:CKBanach}, this case is worth looking at.

\begin{question}
Let $m\in\NN$, $m\ge 2$. Does $\GG^m$ have a unique unconditional basis?
\end{question}

We note that $\GG^m$ can be naturally indentified with $\GG(\VV)$, where $\VV$ is am $m$-dimensional Banach space, and a D+S theory can be developed in this setting. Given $x=(x_n)_{n=1}^\infty$ in $\GG(\VV)$ or in $\GG^*(\VV^*)$ we set $x_n=x(n)$ for all $n\in\NN$. We define an endomorphism $T$ of $\GG(\VV)$ to be diagonal if there is a sequence $(T_n)_{n=1}^\infty$ in $\Bt(\VV)$ such that $T(x)=(T_n(x_n))_{n=1}^\infty$ for all $x=(x_n)_{n=1}^\infty\in\GG(\VV)$. In this terminology, we can use that the vector lattice $\GG(\VV)$ has the D+S property to prove that any complemented unconditional basis $(\xx_k)_{k\in\Kt}$ of $\GG[\VV]$ with projecting functionals $(\xx_k^*)_{k\in\Kt}$ in $\GG^*(\VV^*)$ satisfies
\[
\liminf_k \sup_n \norm{\xx_k^*(n)} \norm{\xx_k(n)} >0.
\]
However, in order to take advantage of Lemma~\ref{lem:lattification}, we should close the gap between this estimate and
\[
\liminf_k \sup_n \abs{\xx_k^*(n) (\xx_k(n)) }>0.
\]
The point is that, at least a priori, $\xx_k^*(n)$ and $\xx_k(n)$ could peak at different coordinates with respect to a basis of $\VV$.

The convexification procedure defines an action of the multiplicative semi-group $[1,\infty)$ on the class of Banach lattices. If we consider $p$-convexified spaces for $0<p<1$, we obtain a group action. The price we pay to obtain this richer structure is widening our scope and going into the study of quasi-Banach lattices.

\begin{question}\label{qt:Gp<1}
Does $\GG^{(p)}$, $0<p<1$, have a unique unconditional basis?
\end{question}

Although by now there are solid foundations on the theory of quasi-Banach lattices thanks mostly to the work of Kalton (see \cites{Kalton1984b,AlbiacAnsorena2025}), some background results may need to be clarified before addressing Question~\ref{qt:Gp<1}. In fact, a perturbation theory for operators between quasi-Banach spaces seems to be missing. In any case, trying to extend to the case $0<p<1$ the proof of Theorem~\ref{thm:CKBanach} verbatim is hopeless. We point out that since the canonical basis of $\GG[p,f]$, $f\in\Ft$, fails to be shrinking in the case when $p<1$, Lemma~\ref{lem:SS:B} does not apply to this basis. We also remark that, in the lack on local convexity, Lemma~\ref{lem:average} does not imply a lower bound for the quasi-norm of some of the vectors we are averaging. Thus, we do not know whether the canonical basis of $\GG[p,f]$, $0<p<1$, has the D+S property.


\begin{bibdiv}
\begin{biblist}
\bib{AlbiacAnsorena2022b}{article}{
author={Albiac, Fernando},
author={Ansorena, Jos\'{e}~L.},
title={On the permutative equivalence of squares of unconditional
bases},
date={2022},
ISSN={0001-8708},
journal={Adv. Math.},
volume={410},
pages={Paper No. 108695, 26},
url={https://doi-org/10.1016/j.aim.2022.108695},
review={\MR{4487973}},
}

\bib{AlbiacAnsorena2022PAMS}{article}{
author={Albiac, Fernando},
author={Ansorena, Jos\'{e}~L.},
title={Uniqueness of unconditional basis of
{$\ell\sp2\oplus\mathcal{T}^{(2)}$}},
date={2022},
ISSN={0002-9939,1088-6826},
journal={Proc. Amer. Math. Soc.},
volume={150},
number={2},
pages={709\ndash 717},
url={https://doi.org/10.1090/proc/15670},
review={\MR{4356181}},
}

\bib{AlbiacAnsorena2022c}{article}{
author={Albiac, Fernando},
author={Ansorena, Jos\'{e}~L.},
title={Uniqueness of unconditional basis of infinite direct sums of
quasi-{B}anach spaces},
date={2022},
ISSN={1385-1292},
journal={Positivity},
volume={26},
number={2},
pages={Paper No. 35, 43},
url={https://doi-org/10.1007/s11117-022-00905-1},
review={\MR{4400173}},
}

\bib{AlbiacAnsorena2024c}{article}{
author={Albiac, Fernando},
author={Ansorena, Jos\'{e}~L.},
title={The structure of greedy-type bases in {T}sirelson's space and its
convexifications},
date={2024},
ISSN={0391-173X,2036-2145},
journal={Ann. Sc. Norm. Super. Pisa Cl. Sci. (5)},
volume={25},
number={3},
pages={1357\ndash 1382},
review={\MR{4855766}},
}

\bib{AlbiacAnsorena2025}{article}{
author={Albiac, Fernando},
author={Ansorena, Jos\'{e}~L.},
title={Isomorphisms between vector-valued {$H_p$}-spaces for $0<p\le 1$ and uniqueness of unconditional structure},
date={2025},
journal={arXiv e-prints},
}

\bib{AADK2021}{article}{
author={Albiac, Fernando},
author={Ansorena, Jos\'{e}~L.},
author={Dilworth, Stephen~J.},
author={Kutzarova, Denka},
title={A dichotomy for subsymmetric basic sequences with applications to {G}arling spaces},
date={2021},
ISSN={0002-9947},
journal={Trans. Amer. Math. Soc.},
volume={374},
number={3},
pages={2079\ndash 2106},
url={https://doi.org/10.1090/tran/8278},
review={\MR{4216733}},
}

\bib{AAW2022}{article}{
author={Albiac, Fernando},
author={Ansorena, Jos\'{e}~L.},
author={Wojtaszczyk, Przemys{\l}aw},
title={Uniqueness of unconditional basis of {$H_p(\mathbb{T})\oplus\ell_2$} and {$H_p(\mathbb{T}) \oplus\mathcal{T}^{(2)}$} for {$0 < p < 1$}},
date={2022},
ISSN={0022-1236},
journal={J. Funct. Anal.},
volume={283},
number={7},
pages={Paper No. 109597, 24},
url={https://doi-org/10.1016/j.jfa.2022.109597},
review={\MR{4447769}},
}

\bib{AlbiacKalton2016}{book}{
author={Albiac, Fernando},
author={Kalton, Nigel~J.},
title={Topics in {B}anach space theory},
edition={Second Edition},
series={Graduate Texts in Mathematics},
publisher={Springer, [Cham]},
date={2016},
volume={233},
ISBN={978-3-319-31555-3; 978-3-319-31557-7},
url={https://doi.org/10.1007/978-3-319-31557-7},
note={With a foreword by Gilles Godefroy},
review={\MR{3526021}},
}

\bib{BL1983}{book}{
author={Beauzamy, B.},
author={Laprest\'{e}, J.-T.},
title={Mod\`{e}les \'{e}tal\'{e}s des espaces de {B}anach},
series={Travaux en Cours. [Works in Progress]},
publisher={Hermann, Paris},
date={1984},
ISBN={2-7056-5965-X},
review={\MR{770062}},
}

\bib{BCLT1985}{article}{
author={Bourgain, Jean},
author={Casazza, Peter~G.},
author={Lindenstrauss, Joram},
author={Tzafriri, Lior},
title={Banach spaces with a unique unconditional basis, up to permutation},
date={1985},
ISSN={0065-9266},
journal={Mem. Amer. Math. Soc.},
volume={54},
number={322},
pages={iv+111},
url={https://doi-org/10.1090/memo/0322},
review={\MR{782647}},
}

\bib{CasKal1998}{article}{
author={Casazza, Peter~G.},
author={Kalton, Nigel~J.},
title={Uniqueness of unconditional bases in {B}anach spaces},
date={1998},
ISSN={0021-2172},
journal={Israel J. Math.},
volume={103},
pages={141\ndash 175},
url={https://doi-org/10.1007/BF02762272},
review={\MR{1613564}},
}

\bib{CasKal1999}{article}{
author={Casazza, Peter~G.},
author={Kalton, Nigel~J.},
title={Uniqueness of unconditional bases in {$c_0$}-products},
date={1999},
ISSN={0039-3223},
journal={Studia Math.},
volume={133},
number={3},
pages={275\ndash 294},
review={\MR{1687211}},
}

\bib{CasShu1989}{book}{
author={Casazza, Peter~G.},
author={Shura, Thaddeus~J.},
title={Tsirelson's space},
series={Lecture Notes in Mathematics},
publisher={Springer-Verlag, Berlin},
date={1989},
volume={1363},
ISBN={3-540-50678-0},
url={https://doi-org/10.1007/BFb0085267},
note={With an appendix by J. Baker, O. Slotterbeck and R. Aron},
review={\MR{981801}},
}

\bib{EdWo1976}{article}{
author={Edelstein, I.~S.},
author={Wojtaszczyk, Przemys{\l}aw},
title={On projections and unconditional bases in direct sums of {B}anach spaces},
date={1976},
ISSN={0039-3223},
journal={Studia Math.},
volume={56},
number={3},
pages={263\ndash 276},
url={https://doi-org/10.4064/sm-56-3-263-276},
review={\MR{425585}},
}

\bib{Gowers1994hyp}{article}{
author={Gowers, William~T.},
title={A solution to {B}anach's hyperplane problem},
date={1994},
ISSN={0024-6093,1469-2120},
journal={Bull. London Math. Soc.},
volume={26},
number={6},
pages={523\ndash 530},
url={https://doi.org/10.1112/blms/26.6.523},
review={\MR{1315601}},
}

\bib{Gowers1994}{article}{
author={Gowers, William~T.},
title={A solution to the {S}chroeder-{B}ernstein problem for {B}anach spaces},
date={1996},
ISSN={0024-6093,1469-2120},
journal={Bull. London Math. Soc.},
volume={28},
number={3},
pages={297\ndash 304},
url={https://doi.org/10.1112/blms/28.3.297},
review={\MR{1374409}},
}

\bib{GowersMaurey1997}{article}{
author={Gowers, William~T.},
author={Maurey, Bernard},
title={{B}anach spaces with small spaces of operators},
date={1997},
ISSN={0025-5831},
journal={Math. Ann.},
volume={307},
number={4},
pages={543\ndash 568},
url={https://doi-org/10.1007/s002080050050},
review={\MR{1464131}},
}

\bib{James1950}{article}{
author={James, Robert~C.},
title={Bases and reflexivity of {B}anach spaces},
date={1950},
ISSN={0003-486X},
journal={Ann. of Math. (2)},
volume={52},
pages={518\ndash 527},
url={https://doi-org/10.2307/1969430},
review={\MR{39915}},
}

\bib{James1964}{article}{
author={James, Robert~C.},
title={Uniformly non-square {B}anach spaces},
date={1964},
ISSN={0003-486X},
journal={Ann. of Math. (2)},
volume={80},
pages={542\ndash 550},
url={https://doi-org/10.2307/1970663},
review={\MR{173932}},
}

\bib{Kalton1984b}{article}{
author={Kalton, Nigel~J.},
title={Convexity conditions for nonlocally convex lattices},
date={1984},
ISSN={0017-0895},
journal={Glasgow Math. J.},
volume={25},
number={2},
pages={141\ndash 152},
url={https://doi-org/10.1017/S0017089500005553},
review={\MR{752808}},
}

\bib{Kato1958}{article}{
author={Kato, T.},
title={Perturbation theory for nullity, deficiency and other quantities of linear operators},
date={1958},
ISSN={0021-7670},
journal={J. Analyse Math.},
volume={6},
pages={261\ndash 322},
url={https://doi-org/10.1007/BF02790238},
review={\MR{107819}},
}

\bib{KotheToeplitz1934}{article}{
author={K{\"o}the, Gottfried},
author={Toeplitz, Otto},
title={Lineare {R}{\"a}ume mit unendlich vielen {K}oordinaten und
{R}inge unendlicher {M}atrizen},
date={1934},
ISSN={0075-4102},
journal={J. Reine Angew. Math.},
volume={171},
pages={193\ndash 226},
url={https://doi-org/10.1515/crll.1934.171.193},
review={\MR{1581429}},
}

\bib{LinPel1968}{article}{
author={Lindenstrauss, Joram},
author={Pe{\l}czy\'{n}ski, Aleksander},
title={Absolutely summing operators in {$L_{p}$}-spaces and their applications},
date={1968},
ISSN={0039-3223},
journal={Studia Math.},
volume={29},
pages={275\ndash 326},
url={https://doi-org/10.4064/sm-29-3-275-326},
review={\MR{0231188}},
}

\bib{LinTza1979}{book}{
author={Lindenstrauss, Joram},
author={Tzafriri, Lior},
title={Classical {B}anach spaces. {II} -- function spaces},
series={Ergebnisse der Mathematik und ihrer Grenzgebiete [Results in Mathematics and Related Areas]}, publisher={Springer-Verlag, Berlin-New York},
date={1979},
volume={97},
ISBN={3-540-08888-1},
review={\MR{540367}},
}

\bib{LinZip1969}{article}{
author={Lindenstrauss, Joram},
author={Zippin, Mordecay},
title={Banach spaces with a unique unconditional basis},
date={1969},
journal={J. Functional Analysis},
volume={3},
pages={115\ndash 125},
url={https://doi-org/10.1016/0022-1236(69)90054-8},
review={\MR{0236668}},
}

\bib{Pel1960}{article}{
author={Pe{\l}czy\'{n}ski, Aleksander},
title={Projections in certain {B}anach spaces},
date={1960},
ISSN={0039-3223},
journal={Studia Math.},
volume={19},
pages={209\ndash 228},
url={https://doi-org/10.4064/sm-19-2-209-228},
review={\MR{126145}},
}

\bib{Schep1992}{incollection}{
author={Schep, Anton~R.},
title={Krivine's theorem and the indices of a {B}anach lattice},
date={1992},
volume={27},
pages={111\ndash 121},
url={https://doi.org/10.1007/BF00046642},
note={Positive operators and semigroups on Banach lattices (Cura\c cao,
1990)},
review={\MR{1184883}},
}

\bib{Schlumprecht1991}{article}{
author={Schlumprecht, Thomas},
title={An arbitrarily distortable {B}anach space},
date={1991},
ISSN={0021-2172,1565-8511},
journal={Israel J. Math.},
volume={76},
number={1-2},
pages={81\ndash 95},
url={https://doi.org/10.1007/BF02782845},
review={\MR{1177333}},
}

\bib{Singer1981}{book}{
author={Singer, Ivan},
title={Bases in {B}anach spaces. {II}},
publisher={Editura Academiei Republicii Socialiste Rom\^{a}nia, Bucharest; Springer-Verlag, Berlin-New York},
date={1981},
ISBN={3-540-10394-5},
review={\MR{610799}},
}

\bib{Tsirelson1974}{article}{
author={Tsirelson, Boris~S.},
title={It is impossible to embed {$\ell_{p}$} or {$c_{0}$} into an arbitrary {B}anach space},
date={1974},
ISSN={0374-1990},
journal={Funkcional. Anal. i Prilo\v{z}en.},
volume={8},
number={2},
pages={57\ndash 60},
review={\MR{0350378}},
}

\bib{Wojtowicz1988}{article}{
author={W\'{o}jtowicz, Marek},
title={On the permutative equivalence of unconditional bases in {$F$}-spaces},
date={1988},
ISSN={0208-6573},
journal={Funct. Approx. Comment. Math.},
volume={16},
pages={51\ndash 54},
review={\MR{965366}},
}
\end{biblist}
\end{bibdiv}
\end{document}